\documentclass{amsart}[12pt]
\usepackage{amssymb,amsthm}
\usepackage{eucal}
\usepackage[shortlabels]{enumitem}
\usepackage{thmtools,thm-restate}
\usepackage{hyperref}

\setlist[description]{font=\normalfont}

\usepackage{interval}
\intervalconfig{soft  open  fences}
\usepackage{mathtools}

\declaretheorem[numberwithin=section]{theorem}
\declaretheorem[sibling=theorem]{proposition}
\declaretheorem[sibling=theorem]{lemma}

\declaretheorem[sibling=theorem]{claim}

\declaretheorem[style=definition,sibling=theorem]{definition}

\declaretheorem[sibling=theorem]{fact}

\DeclareMathOperator{\cf}{cf}
\DeclareMathOperator{\cof}{cof} 
\DeclareMathOperator{\ot}{ot}

\DeclareMathOperator{\dom}{dom}

\DeclareMathOperator{\tc}{tc}

\DeclareMathOperator{\stem}{stem}
\DeclareMathOperator{\ssucc}{succ}
\DeclareMathOperator{\osucc}{osucc}

\renewcommand{\th}{^\text{th}}
\newcommand{\su}{\hspace{1mm}\text{s.t.}\hspace{1mm}}

\newcommand{\seq}[2]{\langle #1 : #2 \rangle}
\newcommand{\pair}[2]{\langle #1 , #2 \rangle}

\newcommand{\Col}{\textup{Col}}
\newcommand{\Add}{\text{Add}}
\newcommand{\ON}{\textup{ON}}

\newcommand{\GCH}{\textup{\textsf{GCH}}}
\newcommand{\CH}{\textup{\textsf{CH}}}
\newcommand{\ZFC}{\textup{\textsf{ZFC}}}

\newcommand{\PCF}{\textup{\textsf{PCF}}}

\newcommand{\WRP}{\textup{\textsf{WRP}}}

\newcommand{\then}{\implies}

\newcommand{\rest}{\upharpoonright}

\newcommand{\nrest}{\!\rest\!}
\newcommand{\la}{\langle}
\newcommand{\ra}{\rangle}
\newcommand{\such}{\; : \;}  

\newcommand{\concat}{\,{}^\frown{} \,}
\newcommand{\nothing}[1]{}

\newcommand{\M}{\mathbb M}
\renewcommand{\P}{\mathbb P}

\newcommand{\F}{\mathbb F}

\newcommand{\B}{\mathbb B}

\renewcommand{\S}{\mathbb S}
\newcommand{\G}{\mathbb G}
\newcommand{\U}{\mathbb U}
\newcommand{\bL}{\mathbb L}
\newcommand{\E}{\mathbb E}

\newcommand{\pI}{\textup{\textsf{I}}}
\newcommand{\pII}{\textup{\textsf{II}}}

\author{Maxwell Levine}
\author{Heike Mildenberger}
\date{\today}

\title{Good Scales and Non-Compactness of Squares}

\begin{document}

\begin{abstract} Cummings, Foreman, and Magidor investigated the extent to which square principles are compact at singular cardinals.  The first author proved that if $\kappa$ is a singular strong limit of uncountable cofinality, all scales on $\kappa$ are good, and $\square^*_\delta$ holds for all $\delta<\kappa$, then $\square_\kappa^*$ holds. In this paper we will present a strongly contrasting result for $\aleph_\omega$. We construct a model in which $\square_{\aleph_n}$ holds for all $n<\omega$, all scales on $\aleph_\omega$ are good, but in which $\square_{\aleph_\omega}^*$ fails and some weak forms of internal approachability for $[H(\aleph_{\omega+1})]^{\aleph_1}$ fail. This requires an extensive analysis of the dominating and approximation properties of a version of Namba forcing. We also obtain a variation of the main result.\end{abstract}

\maketitle

\section{Introduction}

There is less independence exhibited in the behavior of singular cardinals than there is with regular cardinals. Moreover, the circumstances depend on the cofinality of the singular cardinal. One early example of this phenomenon had to do with the behavior of the continuum function. Magidor proved that the general continuum hypothesis ($\GCH$) can fail for the first time at $\aleph_\omega$ \cite{Magidor1977}, while Silver proved that $\GCH$ cannot fail for the first time at a singular of uncountable cofinality \cite{Silver1975}. Questions in this area often take a form pertaining to compactness: How much do the configurations below a singular cardinal affect the configuration \emph{at} the singular cardinal?

The theory behind these phenomena has developed considerably. Shelah introduced $\PCF$ theory in the late 1980's to study the behavior of singular cardinals like $\aleph_\omega$. Using these revolutionary methods, Shelah was able to obtain surprising $\ZFC$ theorems for the cardinal arithmetic of singular cardinals that do not have analogs for regular cardinals. In the early 2000's, Cummings, Foreman, and Magidor wrote a series of papers connecting $\PCF$ theory to the combinatorial properties of canonical inner models. They focused notably on varieties of good scales, which are the most typical tame objects in $\PCF$ theory, and variants of Jensen's square principle, which embody the combinatorial properties of G{\"o}del's model $L$. In this manner, they established much of the basic language and theoretical tools that continue to be used in the investigations of these objects.

This paper will consider the question of the extent to which good scales can be used to construct variants of the square principle for successors of singular cardinals. There is some evidence in the positive direction: that these principles are to some extent compact. On one hand, Cummings, Foreman, and Magidor proved that if $\square_{\aleph_n}$ holds for all $n<\omega$, then the good points of a scale on $\aleph_\omega$ can be used to construct a square-like sequence of length $\aleph_{\omega+1}$ \cite[Theorem 3.5]{Cummings-Foreman-Magidor2004}. On the other hand, Cummings et al$.$ proved that it is consistent that $\square_{\aleph_n}$ holds for all $n<\omega$ while the canonical principle $\square_{\aleph_\omega}$ fails \cite{Cummings-Foreman-Magidor2003}.\footnote{See from the definition of $\square_\kappa$ below that it is really an assertion about $\kappa^+$.} This result was strengthened by Krueger, who showed that a similar model can be obtained with no good scales on $\aleph_\omega$ \cite{Krueger2013}. Nonetheless, the first author proved that if $\kappa$ is a singular strong limit cardinal of uncountable cofinality such that $\square_\delta^*$ (a weak version of $\square_\delta$) holds for all $\delta<\kappa$, and all scales on $\kappa$ are good, then $\square_\kappa^*$ holds \cite{Levine2022compactness}. Hence it is natural to further investigate the interplay between squares below a singular cardinal and what they imply for the successor of that singular cardinal in the presence of good scales.

We will prove a consistency result as our main theorem, and this will lead us to a careful analysis of a version of Namba forcing. This method was originally used to demonstrate that $\aleph_2$ can be singularized without collapsing $\aleph_1$, and to address questions about Boolean algebras, but later it became apparent that Namba forcing is bound up with the study of singular cardinals as such (see e.g$.$ \cite{Bukovsky-Coplakova1990, Foreman-Magidor1995, Foreman-Todorcevic2005, Cox-Krueger2018}).

For technical reasons that will become clear in the course of the paper, we found it useful to define a variation of the notion of internal approachability:

\begin{definition} Let $\theta$ and $\lambda$ be regular cardinals and let $M \prec H(\theta)$. We say that $M$ is \emph{sup-internally approachable at $\lambda$} if there is a sequence $\seq{M_i}{i<\omega_1}$ of countable sets such that
\begin{enumerate}
\item for all $j<\omega_1$, $\seq{M_i}{i<j} \in M_{j+1} \cap M$,
\item $\sup_{i<\omega_1}\sup(M_i \cap \lambda) = \sup(M \cap \lambda)$.
\end{enumerate}
\end{definition}

Our theorem presents a contrast to the situation with compactness of weak squares for singulars of uncountable cofinality.

\begin{theorem}\label{maintheorem} Assuming the consistency of a cardinal $\kappa$ that is $\kappa^{\omega+1}$-supercompact, it is consistent that there is a model of set theory in which the following are true:

\begin{enumerate}
\item $\aleph_\omega$ is a strong limit,
\item all scales on $\aleph_\omega$ are good,

\item $\square_{\aleph_n}$ holds for all $n<\omega$,

\item $\square_{\aleph_{\omega}}^*$ fails,\footnote{The consistency of the conjunction of the first four points in \autoref{maintheorem} was claimed the first arXiv version of the compactness of weak square paper \cite{Levine2022compactness}, but the proof in the original source was flawed.}
\item there are stationarily-many $N \prec H(\aleph_{\omega+1})$ of cardinality $\aleph_1$ that are not sup-internally approachable at $\aleph_{\omega+1}$.
\end{enumerate}\end{theorem}

The use of a large cardinal assumption is necessary for our result. The failure of $\square_\kappa$ for singular $\kappa$ implies the consistency of substantial large cardinals: Sargsyan showed that it implies the existence of a non-tame mouse \cite{Sargsyan2014}, which in particular, by work of Woodin, implies the consistency of a Woodin cardinal above a strong cardinal \cite[Theorem 15.1]{Steel2008}.

\autoref{maintheorem} depends on the approximation and dominating properties (roughly-stated) of a version of Namba forcing, which we analyze in Section 2. We will also use a simple poset for forcing the existence of good scales in the construction. The interaction of the Namba forcing and the good scale forcing will be the crux of the proof, which we provide in Section 3. Our construction will be shaped in a way that gives an analogy with guessing models (see \cite{Viale-Weiss2010}, so that our work can be connected to more recent research.

In Section 4, we will prove some elaborating results that more or less consider possible variations on \autoref{maintheorem}. First, we show that it as possible to obtain a model in which $\square_{\aleph_n}$ holds for all $n<\omega$ while $\square(\aleph_{\omega+1},\aleph_1)$ fails. This square principle would fail under $\mathsf{PFA}$, so this result indicates possible tension between good scales and square principles.

We are assuming that the reader is familiar with the basics of cardinal arithmetic, forcing, and large cardinals (see \cite{Jech2003}).

\subsection{Basic Combinatorial Notions} 

Here we will define some $\PCF$-theoretic notions and recall some fundamental facts. All definitions and facts due to Shelah \cite{Shelah1994}. For the sake of readability, we will give more recent citations and short proofs where possible.

\begin{definition}\

\begin{enumerate}

\item If $\tau$ is a cardinal and $f,g : \tau \to \ON$, then we write $f <^\ast g$ if there is some $j<\tau$ such that $f(i)<g(i)$ for all $i \ge j$.

\item Given a singular cardinal $\kappa$ and a strictly increasing sequence $\seq{\kappa_i}{i<\cf \kappa}$ of regular cardinals converging to $\kappa$, the \emph{product} $\prod_{i<\cf \kappa}\kappa_i$ is a space such that $f \in \prod_{i<\cf \kappa}\kappa_i$ means $\dom f = \cf \kappa$ and $f(i)<\kappa_i$ for all $i<\cf \kappa$.

\item Given a product $\vec \kappa = \prod_{i<\cf \kappa}\kappa_i$, a sequence $\seq{f_\alpha}{\alpha<\nu}$ is a \emph{scale of length $\kappa^+$} on $\vec \kappa$ if:

\begin{enumerate}

\item for all $\alpha<\kappa^+$, $f_\alpha \in \vec \kappa$;

\item for all $\alpha<\beta<\kappa^+$,  $f_\alpha <^\ast f_\beta$;

\item for all $g \in \vec \kappa$, there is some $\alpha<\kappa^+$ such that $g<^\ast f_\alpha$.

\end{enumerate}

\end{enumerate}
\end{definition}

\begin{fact}\label{always-a-scale} If $\kappa$ is singular of cofinality $\lambda$, then there is a product $\prod_{i<\lambda}\kappa_i$ on $\kappa$ that carries a scale of length $\kappa^+$ \cite[Section 2]{Handbook-Abraham-Magidor}.\end{fact}

\autoref{always-a-scale} is only nontrivial if $2^\kappa > \kappa^+$.

\begin{definition} Fix a product $\prod_{i<\cf \kappa}\kappa_i$ on a singular $\kappa$.

\begin{enumerate}

\item If $\vec{f}=\seq{f_\beta}{\beta<\gamma}$ is a $<^*$-increasing subsequence of $\vec \kappa$, then a function $h$ is an \emph{exact upper bound} (or \emph{eub} for short) of $\vec{f}$ if

\begin{enumerate}

\item for all $\beta<\gamma$, $f_\beta <^* h$,

\item for all $g <^* h$, there is some $\beta< \gamma$ such that $g <^* f_\beta$.

\end{enumerate}

\item Given a $<^*$-increasing sequence $\vec f = \seq{f_\beta}{\beta<\gamma}$ on a product $\prod_{i<\cf \kappa}\kappa_i$, we say that $\alpha \le \gamma$ is \emph{good} if there is some unbounded $A \subset \alpha$ with $\ot A = \cf \alpha$ and some $j<\cf \kappa$ such that for all $i \ge j$, $\seq{f_\beta(i)}{\beta \in A}$ is strictly increasing.

\item If there is a club $D \subset \kappa^+$ such that every $\alpha \in D$ with $\cf \alpha > \cf \kappa$ is a good point of $\vec f$, then $\vec f$ is a \emph{good scale}.

\end{enumerate}\end{definition}

It is an exercise to obtain:

\begin{fact} If $\vec{f}=\seq{f_\beta}{\beta<\gamma}$ is a $<^*$-increasing subsequence of $\vec \kappa$, then an exact upper bound of $\vec{f}$ is in particular a least upper bound.\end{fact}

\begin{fact}\label{goodness-characterized}  Let $\seq{f_\beta}{\beta<\alpha}$ be a sequence of functions in a product $\prod_{i<\cf \kappa}\kappa_i$ where $\cf \kappa < \kappa_0$. Then the following are equivalent if $\cf(\alpha)>\cf(\kappa)$:

\begin{enumerate}
\item $\alpha$ is a good point.
\item There is a $<$-increasing sequence $\seq{h_\gamma}{\gamma < \cf(\alpha)}$ such that:

\begin{enumerate}

\item for all $i<\cf \kappa$ and $\gamma<\gamma'$, $h_\gamma(i)<h_{\gamma'}(i)$, 

\item for all $\gamma < \cf(\alpha)$, there is $\beta<\alpha$ such that $h_\gamma <^* f_\beta$, and for all $\beta<\alpha$, there is $\gamma<\cf(\alpha)$ such that $f_\beta<^* h_\gamma$. 

\end{enumerate}

\item There is an exact upper bound $h$ of $\seq{f_\beta}{\beta<\alpha}$ such that for some $j<\cf \kappa$, $\cf(h(i))=\cf(\alpha)$ for $i \ge j$.
\end{enumerate}
\end{fact}

\begin{proof} (See \cite[Lemma 2.1]{Cummings-Foreman-Magidor2004}, \cite[Section 13]{Cummings2005}.)

For \emph{(1) $\Rightarrow$ (2):}, let $A \subseteq \alpha$ and $j<\cf \kappa$ witness goodness. For each $\gamma<\cf(\alpha)$, let $h_\gamma(i):=\sup \{f_\beta(i):\beta \in A, \ot(A \cap \beta)<\gamma\}$.  For \emph{(2) $\Rightarrow$ (3):} Given $\seq{h_\gamma}{\gamma<\cf(\alpha)}$ as in \emph{(2)}, let $h(i):= \sup_{\gamma<\cf(\alpha)}h_\gamma(i)$. For \emph{(3) $\Rightarrow$ (2):} take such an exact upper bound and observe that by the assumption $\cf \kappa < \kappa_0$ we can assume that $\cf(h(i))=\cf(\alpha)$ for \emph{all} $i<\cf \kappa$. Let $\seq{\beta^i_\xi}{\xi<\cf(\alpha)}$ be cofinal in $h(i)$ for $i<\cf\kappa$. Then let $h_\xi:i \mapsto \beta^i_\xi$.

\emph{(2) $\Rightarrow$ (1):}\footnote{This part is often known as the ``sandwich argument.''} Fix such a $\seq{h_\gamma}{\gamma<\cf(\alpha)}$. Choose $\seq{\beta_\xi,\gamma_\xi}{\xi<\cf(\alpha)}$ cofinal such that for all $h$, $h_{\gamma_\xi} <^* f_{\beta_\xi} <^* h_{\gamma_{\xi+1}}$. For all $\xi$, let $j_\xi$ be such that for all $i \ge j_\xi$, $h_{\gamma_\xi}(i) < f_{\beta_\xi}(i) < h_{\gamma_{\xi+1}}(i)$. There is some $j$ and some unboudned $A' \subseteq \cf(\alpha)$ such that for all $\xi \in A'$, $j_\xi = j$. Then $A:=\{\beta_\xi:\xi \in A\}$ and $j$ witness goodness for $\alpha$.\end{proof}

\autoref{goodness-characterized} is particularly useful because of the uniqueness of exact upper bounds:

\begin{fact}\label{unique-eubs} If $g$ and $h$ are exact upper bounds of $\seq{f_\beta}{\beta<\alpha}$, then $g=^* h$.\end{fact}

\begin{proof} Suppose $g$ and $h$ are eub's of $\seq{f_\beta}{\beta<\alpha}$ in a product $\prod_{i<\cf \kappa}\kappa_i$ but that $g \not\le^* h$, so there is an unbounded $X \subseteq \cf \kappa$ such that for all $i \in X$, $h(i)<g(i)$. Let $h'(i)( = h(i)$ for $i \in X$ and $h'(i)=0$ for $i \notin X$. Then $h' <^* g$, so there is some $\beta<\alpha$ such that $h' <^* f_\beta$ since $g$ is an exact upper bound. If $j$ is such that $i>j$ implies $h'(i)<^* f_\beta(i)$, then this means that for all $i \in X \setminus j$, $h(i) < f_\beta(i)$, so $f_\beta <^* h$, contradicting that $h$ is an eub.\end{proof}

We will also give a definition of the square principle.

\begin{definition} If $\kappa$ is a cardinal, we say the $\square_\kappa$ holds if there is a sequence 
$\seq{C_\alpha}{\alpha<\kappa^+}$ such that the following hold for all $\alpha<\kappa^+$:

\begin{enumerate}
\item $C_\alpha$ is a club in $\alpha$,
\item for all $\beta \in \lim C_\alpha$, $C_\alpha \cap \beta = C_\beta$,
\item $\ot(C_\alpha) \le \kappa$.
\end{enumerate}
\end{definition}

We also make some use of variations in internal approachability. These variations were introduced by Foreman and Todor{\v c}evi{\' c} \cite{Foreman-Todorcevic2005} and were proved distinct by Krueger \cite{Krueger2007,  Krueger2008}.

\begin{definition} Given an uncountable regular $\kappa$ and a set $N \in [H(\theta)]^\kappa$, we say: 

\begin{itemize}

\item $N$ is \emph{internally unbounded} if $\forall x \in P_\kappa(N), \exists M \in N, x \subseteq M$,
\item $N$ is \emph{internally stationary} if $P_\kappa(N) \cap N$ is stationary in $P_\kappa(N)$,
\item $N$ is \emph{internally club} if $P_\kappa(N) \cap N$ is club in $P_\kappa(N)$,
\item $N$ is \emph{internally approachable} if there is an increasing and continuous chain $\seq{M_\xi}{\xi<\kappa}$ such that $|M_\xi|<\kappa$ and $\seq{M_\eta}{\eta<\xi} \in M_{\xi+1}$  for all $\xi<\kappa$ such that $N=\bigcup_{\xi<\kappa} M_\xi$.

\end{itemize}
\end{definition}

\begin{definition}\label{tightness-def} A model $M \prec H(\theta)$ is \emph{tight} with respect to $K=\seq{\kappa_i}{i \in I}$ if $M \cap \prod_i \kappa_i$  is cofinal in $\prod_i (\kappa_i \cap M)$ in the $<^*$-ordering. \end{definition}

\subsection{The Namba Forcing} Here we will define the most important part of the proof of \autoref{maintheorem}.

\begin{definition} Let $T$ be a tree.
\begin{enumerate}

\item  
For an ordinal $\alpha$, the set
$T_\alpha$ is the set of $t \in T$ with $\dom(t) =\alpha$.

\item  
The \emph{height} $\operatorname{ht}(T)$ of a tree $T$ is $\min\{\alpha : T(\alpha) = \emptyset\}$.

\item 

We let $[T] = \{f \colon \operatorname{ht}(T) \to \kappa : \forall \alpha < \operatorname{ht}(T), f \rest \alpha \in T\}$. 
Elements of $[T]$ are called \emph{cofinal branches}.
\item 
For
$t_1$, $t_2 \in T \cup [T]$ we write $t_1 \sqsubseteq t_2$ if $t_2 \rest \dom(t_1) = t_1$. The tree order is the relation $\sqsubseteq$. If $t = s \cup \{(\dom(s), \beta)\}$, we write $t = s {}^\frown \langle \beta \rangle$.

\item 
$T \nrest \alpha = \bigcup_{\beta<\alpha}T_\beta$.

\item 
$T \rest t = \{s \in T : s \sqsubseteq t  \vee t \sqsubseteq s\}$.

\item 
For $t \in T_\alpha$ we let $\ssucc_T(t) = \{c : c \in T_{\alpha+1} \wedge c \sqsupseteq t\}$ denote the \emph{set of immediate successors of $t$}, and $\osucc_T(t) = \{\beta : t {}^\frown \langle \beta \rangle  \in T_{\alpha+1}\}$ denote the \emph{ordinal successor set of $t$}.

\item We call $t \in T$ a \emph{splitting node} if $|\ssucc_T(t)|>1$.

\item $\stem(T)$ is the $\sqsubseteq$-minimal splitting node. 

\end{enumerate}
\end{definition}

\begin{definition}\label{laver-namba-def-rep} Fix a function $d: \omega \to \omega \setminus \{0,1\}$ such that:

\begin{enumerate}
\item For all $m \ge 2$, there are infinitely many $n$ such that $d(n) = m$;
\item If $n$ is the least number such that $d(n)=m$, then for all $k<n$, $d(k)<m$.
\end{enumerate}

The poset $\bL$ will consist of conditions $p$ such that the following hold:

\begin{enumerate}

\item $p$ is a tree consisting of finite sequences $t$.

\item For all $t \in p$ and $n \in \dom(t)$, $t(n) \in \aleph_{d(n)}$.

\item Let $t \in p$ be the unique node maximal in the ordering of $p$ such that for all $s \in p$, either $t \sqsubseteq s$ or $s \sqsubseteq t$. Then for all $t \in p$ with $t \sqsupseteq \stem(p)$, if $n = \dom(t)$, then $\{\eta:  t {}^\frown \eta \in p \}$ is a stationary subset of $\aleph_{d(n)} \cap \cof(\omega_1)$.

\end{enumerate}

The ordering on $\bL$ is given by inclusion: $p \le q$ (i.e$.$ $p$ contains more information than $q$) if and only if $p \subseteq q$.\end{definition}

\begin{definition}\label{what-stems-are} If $t \in p$ is the unique node such that for all $s \in p$, either $t \sqsubseteq s$ or $s \sqsubseteq t$, then $t$ is called the \emph{stem} of $p$ and is denoted $\stem(p)$.\end{definition}

We also have the requisite notion of fusion, which will be familiar to readers who have seen tree forcings.

\begin{definition} If $p \in \bL$, we write $n(p) := |\stem(p)|$. If $S, p \in \bL$ and $n<\omega$, we write $q \le_n p$ if $q \le p$, $\stem(q) = \stem(p)$, and for all $t$ with $|t| \le n(q) + n$, $t \in q$ if and only if $t \in p$.

We say that $\seq{p_n}{n<\omega}$ is a \emph{fusion sequence} if $p_n \ge_n p_{n+1}$ for all $n<\omega$.\end{definition}

\begin{fact} If $\seq{p_n}{n<\omega}$ is a fusion sequence of conditions in $\bL$, then $\bigcap_{n<\omega} p_n \in \bL$.\end{fact}

Our poset is similar to a number of singular Namba forcings that appear in the literature \cite{Bukovsky-Coplakova1990,Cummings-Foreman-Magidor2003,Krueger2013}, but the particular properties of $\bL$ will be important for the proof of \autoref{maintheorem}.

First, there is the fact that $\bL$ is a ``Laver-style'' poset in which there is one stem as in \autoref{what-stems-are}. In fact, a similar model of Krueger has a non-Laver style Namba forcing in order \emph{not} to have any good scales \cite{Krueger2013}. This is necessary for \autoref{cummings-magidor-strong-decision-cdot} below and will be used in \autoref{section-eubs} below to derive exact upper bounds that are added by the forcing. In \autoref{section-eubs}, we will use the fact that the splitting sets are stationary rather than merely cofinal because we need a normal ideal for an application of Fodor's Lemma. The function $d$ is used to ensure that the $\aleph_n$'s are singularized to have cofinality $\omega$, which will be used to apply idea of Cummings et al$.$ to get $\square_{\aleph_n}$-sequences in the final model. Finally, we need our forcing to split into sets concentrating on cofinality $\omega_1$, because the exact upper bounds that are added must stabilize to output points of cofinality $\omega_1$, so that the characterization of goodness from \autoref{goodness-characterized} can be applied.

Now we can collect some properties of our Namba forcing for which existing arguments suffice without alteration.

\begin{fact} For all $n$ such that $1<n<\omega$, $\bL$ forces that $\cf(\aleph_n^V)=\omega$.\end{fact}

\begin{proof} Observe that for all $m<\omega$ and $t \in T$, there are infinitely many $n$ with $d(m)=n$ such that for some $t' \sqsupseteq t$, $\{\eta:t' {}^\frown \eta \in T\}$ has cardinality $\aleph_m$. This then comes from the fact that there are infinitely many $k$ such that $d(k)=n$: A genericity argument defines a cofinal function whose domain consists of these $k$'s.\end{proof}

\begin{fact}\label{namba-succ-sing-pres} $\bL$ forces that $\aleph_{\omega+1}^V$ is an ordinal of cardinality and cofinality $\ge \aleph_1$.\end{fact}

\begin{proof}[Sketch of Proof] This comes from a fusion argument using e.g$.$ \autoref{cummings-magidor-prikry-cdot} below, where we build $p' \le p$ with at most $|p'|=\aleph_\omega$-many possible decisions for $\dot \alpha$.\end{proof}

\begin{fact}\label{namba-omega1-pres} Then $\bL$ preserves stationary subsets of $\aleph_1$.\end{fact}

This is a variation of the arguments presented by Cummings-Magidor \cite{Cummings-Magidor2011} and Krueger \cite{Krueger2013} using an open game and the fact that the splitting nodes all split into sets of size $>\aleph_1$. The component of \autoref{mini-approx-thm-laver} below starting with \autoref{game-claim} is a variation of this argument.

\section{Some Technical Ideas Needed for the Proof}

This section will present the main technical ideas that are more or less new to this paper.

\subsection{Some Approximation for the Namba Forcing}\label{namba-approx-sec}

The first technical idea that we will discuss is an approximation-like result that holds for our forcing $\bL$. The approximation property originates in work of Hamkins \cite{Hamkins2001} and often comes up when obtaining various so-called compactness properties, like the tree property, square principles, and so on.

\begin{theorem}\label{mini-approx-thm-laver} Assume $2^{\aleph_\omega}=\aleph_{\omega+1}$. Let $\dot{\U}$ be a $\P$-name for a countably closed forcing. Then if $\dot{F}$ is a $\bL\ast \dot{\U}$-name for a function with domain $\omega_1$ that is cofinal in $\nu$ where $\cf^V(\nu) \ge \aleph_{\omega+1}^V$, then $\bL \ast \dot{\U}$ forces that there is some $i<\omega_1$ such that $\dot{F} \rest i  \notin V$.\end{theorem}

We are formulating the lemma as such---in terms of a two-step iteration---so it fits the analog that is already in the literature \cite{Levine2023b}. This would be useful for applications where some guessing before is sought for substructures of some $H(\theta)$ where $\theta>\aleph_{\omega+1}$. The point is that $\bL$ would not collapse $H(\theta)^V$ on its own, and would be paired with a L{\'e}vy collapse.

Since we are dealing with two-step iterations we will write $(p,\dot{c}) \le_0 (q,\dot{d})$ if $(p,\dot{c}) \le (q,\dot{d})$ and $p \le_0 q$.

We need versions of the Cummings-Magidor facts that account for an iteration following the initial Namba forcing.

\begin{proposition}\label{cummings-magidor-strong-decision-cdot} Let $\dot{\U}$ be a $\bL$-name for a forcing poset. Suppose that $(p,\dot{c}) \in \bL \ast \dot{\U}$ and that $\dot{\delta}$ is a name for an ordinal below $\omega_1$. Then there is some $(q,\dot{d}) \le_0 (p,\dot{c})$ such that $(q,\dot{d})$ decides a value for $\dot{\delta}$.\end{proposition}

The proof is essentially the same as in the version of the lemma omitting $\dot{\U}$, except that at a particular step we use a gluing argument.

\begin{proof}[Proof of \autoref{cummings-magidor-strong-decision-cdot}] Suppose that the proposition is false. We say that a node $t$ is \emph{bad} if there is no $(q,\dot{d}) \le (p,\dot{c})$ with $q \le_0 p \rest t$ such that $(q,\dot{d})$ decides a value for $\dot{\delta}$. Hence, we are working under the assumption that $\stem p$ is bad.

We will construct a fusion sequence through the following: If $t$ is bad and $I$ is the ideal that defines the splitting for $t$, then the set of $\alpha \in \osucc_p(t)$ such that $t {}^\frown \la \alpha \ra$ is bad is $I$-positive. Otherwise we have a set $W \in I^+$ such that $W \subseteq \osucc_p(t)$ and all $\alpha \in W$ are not bad. Then for each $\alpha \in W$ we choose some $(q_\alpha,\dot{d}_\alpha) \le (p \rest t,\dot{c})$ with $q_\alpha \le_0 p \rest t$ deciding a value $\beta_\alpha$ for $\dot{\delta}$. By $\omega_2$-completeness, there is some $W' \in I^+$ with $W' \subseteq W$ and some $\gamma<\tau$ such that for all $\alpha \in W'$, $\beta_\alpha = \gamma$. Now let $q = \bigcup_{\alpha \in W'}q_\alpha$ and let $\dot{d}$ be the name that glues together the $\dot{d}_\alpha$'s below $q$. Then $(q,\dot{d})$ forces that $\dot{\delta}=\gamma$ and $q \le_0 p \rest t$, hence $t$ is not in fact bad.

We are then able to construct some $r \le_0 p$ such that all $t \in r$ are bad. Find some $(r',\dot{d}) \le (r,\dot{c})$ deciding a value for $\dot{\delta}$. If $s = \stem r'$, then this contradicts the fact that $t \in r'$ is bad.\end{proof}

\begin{proposition}\label{cummings-magidor-prikry-cdot} Let $\dot{\U}$ be a $\bL$-name for a forcing poset. Suppose $(p,\dot{c}) \in \P \ast \dot{\U}$ and suppose $\dot{x}$ is a name for an element of $V$. There is some $(q,\dot{d}) \le_0 (p,\dot{c})$ and some $h_n$ such that for all $t \in q$ with $|t|=h_n$, $(q \rest t,\dot{d})$ decides a value for $\dot{x}$.\end{proposition}

\begin{proof} We call a node $t \in p$ \emph{bad} if it not the case that there is $(q,\dot{d}) \le (p,\dot{c})$ with $q \le_0 p \rest t$ and $n<\omega$ such that every $s \in q$ with $|s|=n$ is such that $(q \rest s,\dot{d})$ decides a value for $\dot{x}$.

We argue that if $t$ is bad and $I$ is its splitting ideal, then the set of $\alpha \in \osucc_p(t)$ such that $t {}^\frown \la \alpha \ra$ is bad is $I$-positive. Otherwise there is an $I$-positive $W \subseteq \osucc_p(t)$ such that for all $\alpha \in W$, there is $q_\alpha \le_0 p \rest t$, $n_\alpha$, and $\dot{d}_\alpha$ such that for every $s \in q_\alpha$ with $|s|=n_\alpha$, $(q_\alpha \rest s,\dot{d}_\alpha)$ decides $\dot{x}$. We find some $I$-positive $W' \subseteq W$ and some $n<\omega$ such that for all $\alpha \in W'$, $n_\alpha = n$. Then let $q = \bigcup_{\alpha \in W'}q_\alpha$ and let $\dot{d}$ be the gluing of the $\dot{d}_\alpha$'s below $q$. Then we can see that $q \le_0 p \rest t$, and hence that $t$ cannot be bad.

Then we obtain an overall contradiction in a manner similar to the proof of \autoref{cummings-magidor-strong-decision-cdot}.\end{proof}

This proof uses ideas from the result that the classical version of Namba forcing (from Jech's textbook \cite[Chapter 28]{Jech2003}) consistently has the weak $\omega_1$-approximation property \cite{Levine2023b}, some material from which is to some extent repeated here. The main changes are that we must alter the statement to suit our situation, and that we must apply it to a Laver-style Namba forcing.

\begin{proof} Suppose for contradiction that $\dot{F}:\omega_1 \to \nu$ is an $\bL$-name for a (necessarily) new cofinal subset of $\nu$ of order-type $\omega_1$ all of whose proper initial segements are in $V$.

Let $\varphi(i,q,\dot{d})$ denote the formula
\begin{align*}
 i < \omega_1 \wedge & (q,\dot{d}) \in \bL \ast \dot{\U}   \wedge \exists \seq{A_\alpha}{\alpha \in \osucc_q(\stem(q))} \su \\
&\forall \alpha \in \osucc_q(\stem(q)),(q \rest (\stem(q) {}^\frown \langle \alpha \rangle),\dot{d}) \Vdash ``\dot{F} \rest i \in A_\alpha\textup{''} \wedge \\
& \forall \alpha, \beta \in \osucc_q(\stem(q)), \alpha \ne \beta \then A_\alpha \cap A_\beta = \emptyset.
\end{align*}

\begin{claim}\label{basic-claim}  $\forall j<\omega_1, (p,\dot{c}) \in \bL$, there is some $i \in (j,\omega_1)$ and some $(q,\dot{d}) \le_0 (p,\dot{c})$ such that $\varphi(i,p,\dot{c})$ holds.\end{claim}

\begin{proof} Let $W = \osucc_p(\stem(p))$. By induction on $\alpha \in W$ we will define a sequence of conditions, $\seq{(q_\alpha,\dot{d}_\alpha)}{\alpha \in W}$, a sequence of natural numbers $\seq{n_\alpha}{\alpha \in W}$, a sequence of countable ordinals $\seq{i_\alpha}{\alpha \in W}$, and the sets $\seq{A_\alpha}{\alpha \in W}$ of cardinality strictly less than $\nu$. After these objects are defined, we will finalize a choice of $(q,\dot{d})$ and the $A_\alpha$'s.

If $\alpha = \min W$, then we can choose an arbitrary $i_\alpha \in (j,\omega_1)$. We apply \autoref{cummings-magidor-prikry-cdot} to find some $(q_\alpha,\dot{d}_\alpha) \le_0 (p \rest (\stem p {}^\frown \langle \alpha \rangle),\dot{c})$ and some $n_\alpha \in \omega$ such that for all $s \in q_\alpha$ with $|s|=n_\alpha$, $(q_\alpha,\dot{d}_\alpha)$ decides $\dot{F} \rest i_\alpha$. Then we let $A_\alpha = \{a:\exists t \in q_\alpha) \su |t|=n_\alpha \wedge (q_\alpha \rest t,\dot{d}_\alpha) \Vdash \textup{``}\dot{F}\rest i_\alpha = a \textup{''}\}$. (Establishing this case is just a formality.)

Now suppose that the members of our sequences have been defined for $\beta \in W \cap \alpha$. Let $B = \bigcup_{\beta \in \alpha \cap W}A_\beta$, which is in particular of cardinality strictly less than $\nu$. Then observe that
\[
(p \rest t \concat \la \alpha \ra,\dot{c})  \Vdash \textup{``}\{\dot{F} \rest i : i \in (j,\omega_1)\}  \not \subseteq B \textup{''}
 \]
since otherwise there would be some $(q,\dot{d}) \le (p \rest t {}^\frown \la \alpha \ra,\dot{c})$ such that $(q,\dot{d}) \Vdash \textup{``}\{\dot{F} \rest i : i \in (j,\omega_1)\} \subseteq B \textup{''}$. We know that $\dot{F} \rest i$ is forced to be bounded in $\nu$: If $\nu = \aleph_{\omega+1}^V$ then this is by \autoref{namba-succ-sing-pres}, otherwise it follows using the chain condition and our assumption that $2^{\aleph_\omega}=\aleph_{\omega+1}$. It is therefore implied that $(q,\dot{d})$ forces $\dot{F}$ to be bounded in $\nu$, and this is a contradiction of our assumptions.

This means that
\[
(p \rest t \concat \la \alpha \ra,\dot{c}) \Vdash \textup{``}\exists i \in (j,\omega_1),\dot{F} \rest i  \notin B \textup{''}.
 \]
So we let $\dot{k}$ be the $\P$-name for the ordinal in $(j,\omega_1)$ witnessing this expression. By \autoref{cummings-magidor-strong-decision-cdot}, there is some $(q'_\alpha ,\dot{d}_\alpha') \le_0 (p \rest t {}^\frown \la \alpha \ra,\dot{c})$ and some $i_\alpha$ such that $(q'_\alpha,\dot{d}_\alpha') \Vdash \textup{``}\dot{k} = i_\alpha \textup{''}$. Apply \autoref{cummings-magidor-prikry-cdot} to find some $n_\alpha \in \omega$ and some $(q_\alpha,\dot{d}_\alpha) \le_0 (q'_\alpha,\dot{d}_\alpha')$ such that for any $t \in q_\alpha$ of height $n_\alpha$, $(q_\alpha \rest t,\dot{d}_\alpha)$ decides $\dot{F} \rest i_\alpha$. Then (as in the base case) we let $A_\alpha = \{a:\exists t \in (q_\alpha)_{n_\alpha},(q_\alpha \rest t,\dot{d}_\alpha) \Vdash \textup{``}\dot{F} \rest i_\alpha = a\textup{''}\}$. Observe that $A_\alpha \cap A_\beta = \emptyset$ for all $\beta<\alpha$.

Let $I$ be ideal in the sequence of $\mathcal{I}$ such that $\bL$ is defined to have $I$-positive splitting for $\osucc_p(t)$. Finally, we choose some $W' \subseteq W$ that is $I$-positive and such that there is some $i$ such that for all $\alpha \in W'$, $i_\alpha = i$. Then let $q = \bigcup_{\alpha \in W'}q_\alpha$. Let $\dot{d}$ be the gluing of the $\dot{d}_\alpha$'s below $q_\alpha$.\end{proof}

We plan to build a fusion sequence using \autoref{basic-claim}. For this purpose we define a game $\mathcal{G}_k$ for $k<\omega_1$.

Suppose round $n$ of the game is being played where $n=0$ is the first round. If $n=0$ then let $(q_*,\dot{d}_*)$ be the starting condition $(p_{-1},\dot{d}_{-1})$ where $|\stem(p_{-1})|=m$ and let $i_* = 0$. Otherwise if $n>0$ let $(q_*,\dot{d}_*,i_*)$ be $(q_{n-1},i_{n-1})$. First Player $\pI$ chooses a subset $Z_n \subseteq \osucc_{q_*}(\stem(q_*))$ with $Z_n \in I_{m+n}$ and some $\delta_n<k$. Then Player $\pII$ chooses some $\alpha \in \osucc_{q_*}(\stem(q_*)) \setminus Z_n$ and some condition $(q_n,\dot{d}_n)\le_0 (q_* \rest \stem q_* {}^\frown \langle \alpha \rangle ,\dot{d}_*)$ and some $i_n \in (\delta_n,k)$ such that $\varphi(q_n,\dot{d}_n,i_n)$ holds. Hence we have the following diagram:

\begin{center}
\def\arraystretch{1.25}
\begin{tabular}{c | c | c | c | c | c | c | c} 
  \textup{Player \textsf{I}} & $Z_0,\delta_0$ &  & $Z_1,\delta_1$ &  & $Z_2,\delta_2$ & & \ldots \\ 
 \hline
  \textup{Player \textsf{II}} &  & $q_0,\dot{d}_0,i_0$ &  & $q_1,\dot{d}_1,i_1$ &  & $q_2,\dot{d}_2,i_2$ & \ldots
\end{tabular}
\end{center}

Player $\textsf{\textup{II}}$ loses at stage $n$ if they cannot an appropriate pair $(q_n,\dot{d}_n,i_n)$ witnessing $\varphi(i_n,q_n,\dot{d}_n)$ for some $i_n<k$, i.e$.$ if they cannot in particular find such $i_n \in (\delta_n,k)$. Otherwise, if Player $\pII$ does not lose at any finite stage, then Player $\pII$ wins.

The following can be proved using standard arguments for Namba-style games (see \cite{Namba1971},\cite[Claim 10]{Levine2023b},\cite[Fact 5]{Cummings-Magidor2011}, \cite[Proposition 3.4]{Krueger2013}).
	
\begin{claim}\label{game-claim} For some $k<\omega_1$, Player $\textup{\textsf{II}}$ has a winning strategy in $\mathcal{G}_k$.\end{claim}

\begin{proof}[Sketch of Proof]The essential idea is the following: By the Gale-Stewart Theorem, the failure of the claim implies that for all $i<\omega_1$ is a winning strategy $\sigma_i$ for Player $\textup{\textsf{I}}$ in $\mathcal{G}_i$. Then take an elementary submodel $M \prec H(\theta)$ with $\seq{\sigma_i}{i<\omega_1} \in M$. Then it is possible to construct a run of the game $\mathcal{G}_k$ such that Player $\pI$ uses the strategy $\sigma_k$ but nonetheless loses the game, and this is done by ensuring that Player $\pII$'s moves are all in $M$ even though $\sigma_k \notin M$. The crux of the argument is that it is possible to take the union of $\aleph_1$-many sets in the relevant ideal to obtain a set in that ideal.\end{proof}

Now we will build a condition $q \in \bL$ by a fusion process in such a way that any stronger condition deciding $\dot{F} \rest k$ will also code the generic sequence for $\bL$.

Fix a sequence $\seq{\delta_n}{n<\omega}$ converging to $k$. Let $p_0 = \bar{p}$ be the starting point where $|\stem(\bar{p})|=m$. We will define a fusion sequence $\seq{p_n}{n < \omega}$ and a sequence $\seq{\dot{c}_n}{n<\omega}$ by induction on $n<\omega$ in such a way that $p_{n+1} \Vdash \textup{``}\dot{c}_{n+1} \le \dot{c}_n \textup{''}$ and such that:

\begin{quote} For all $n<\omega$, then for all $t \in p_n$ with $|t|=m+n$, the following is the case: Let $s_0 \sqsubseteq s_1 \sqsubseteq \ldots \sqsubseteq s_n=t$ be the sequence of all nodes up to and including $t$. Then there is a sequence $Z^t_0,\ldots,Z^t_n$ such that
\[
(Z^t_0,\delta_0),(p_0 \rest s_0,\dot{c}_0,i_0),\ldots,(Z^t_n,\delta_n),(p_n \rest s_n,\dot{c}_n,i_n)
\]
is a run of the game $\mathcal{G}_k$ in which Player $\pII$'s moves are determined by the winning strategy obtained in \autoref{game-claim}.\end{quote}

Note that the third point implies the following: For all positive $n<\omega$, for all $t \in p_n$ with $|t|=m+n$, there is $i_t \in (\delta_n,k)$ and a sequence $\seq{A_s}{s \in \ssucc_{p_n}(t)}$ witnessing that $\varphi(i_t,p_n \rest t,\dot{c}_n)$ holds.

We construct the fusion sequence as follows: Start with stage $-1$ for convenience and let $p_{-1}=p$. Now assume we have defined $p_{n-1}$, and we are considering $t \in p_{n-1}$ with $|t|=m+n$. Let $s_0 \sqsubseteq s_1 \sqsubseteq \ldots \sqsubseteq s_{n-1} = t$ be the sequence of splitting nodes up to and including $t$. Let $S_t$ be the set of $\alpha \in \osucc_{p_{n-1}}(t)$ such that for some $Z^\alpha_n$, the winning strategy for Player $\pII$ applied to the sequence
\[
(Z^t_0,\delta_0),(p_0 \rest s_0,\dot{c}_0,i_0),\ldots,(Z^t_{n-1},\delta_{n-1}),(p_{n-1} \rest s_{n-1},,\dot{c}_{n-1},i_{n-1}),(Z^\alpha_n,\delta_n)
\]
produces some $(q_n,,\dot{d}_n,i_n)$ where $(q_n,\dot{d}_n) \le_0 (p_{n-1} \rest t {}^\frown \langle \alpha \rangle,\dot{c}_{n-1})$. We claim that $|S_t| \notin I_{m+n}$. Otherwise Player $\pI$ would have a winning move for the sequence
\[
(Z^t_0,\delta_0),(p_0 \rest s_0,\dot{c}_0,i_0),\ldots,(Z^t_{n-1},\delta_{n-1}),(p_{n-1}\rest
s_{n-1},\dot{c}_{n-1},i_{n-1})
\]
by playing $S_t$ as the $I_{m+n}$-component of their move. For each such $t $ and $\alpha \in S_t$, choose $q_{t,\alpha}$ to be produced by the winning strategy for Player $\pII$ as the $\bL$-component of their move. Now let $p_n = \bigcup \{q_{t,\alpha}:|t|=m+n,\alpha \in S_t\}$.

Now let $q$ be the fusion limit of $\seq{p_n}{n<\omega}$ and let $\dot{d}$ be the $\bL$-name for the lower bound of $\seq{\dot{c}_n}{n < \omega}$. Then $(q,\dot{d})$ forces that the generic sequence for $\P$ can be recovered from $\dot{F} \rest k$ as follows: Let $(r,\dot{e}) \le (q,\dot{d})$ force $\dot{F} \rest  k = g \in V$. We can inductively choose a cofinal branch $b \subset r$ such for all $t \in b$, for some $i_t<k$, $g \rest i_t = a_t$. Specifically, we construct $b$ by defining a sequence $\seq{s_n}{n<\omega}$ of splitting nodes as follows: Let $s_0 = \stem r$. Given $s_n$, let $s^*_{n+1} \sqsupseteq s_n$ be a direct successor of $s_n$ in $r$ . Then since $\varphi(i_t,r \rest s^*_{n+1},\dot{e})$ holds for some $i_t \in (\delta_n,k)$, there is some $\alpha \in \osucc_r(s^*_{n+1})$ such that $(r \rest s^*_{n+1} {}^\frown \langle \alpha \rangle,\dot{e}) \Vdash \textup{``}g \rest i_t \in  A_t \textup{''}$. Then let $s_{n+1} = s_{n+1}^* {}^\frown \langle \alpha \rangle$. Then let $b = \{t \in r: \exists n<\omega, t \sqsubseteq s_n\}$. This implies that $(r,\dot{e})$ forces that the generic object is equal to $b$, i.e$.$ that $\bigcap \Gamma(\P)=b \in V$, but this is not possible.

Hence $(q,\dot{d}) \Vdash ``\dot{F} \rest k \notin V\textup{''}$ lest we obtain the contradiction from the previous paragraph. This contradicts the premise from the beginning of the proof that initial segments of $\dot{F}$ are in $V$.\end{proof}

\subsection{Some Exactness of Upper Bounds}\label{section-eubs}

Because the forcing we use is meant to provide a master condition for the forcings adding the $\square_{\aleph_n}$'s (which is not needed in Cummings-Magidor \cite{Cummings-Magidor2011}), we must make some adjustments to their arguments.

We want to examine the interaction of this repeating version with scales from the ground model.

\begin{definition}\label{generic_min_function}
We isolate two particular names.
\begin{enumerate}
\item
Let $\dot{b}_{\textup{\textsf{full}}}$ be a $\bL$-name for the \emph{generic branch}, meaning if $G$ is $\bL$-generic, then $\dot{b}_{\textup{\textsf{full}}}$ evaluates to
\[
b_{\textup{\textsf{full}}} = \bigcup\{ \stem(p) \such p \in G \}.
\]
\item
Recall the function $d$ from the definition of $\bL$. We write $d_{\rm min}^{-1}(m) $ for the
minimal $n$ with $d(n) = m$.

\item Let $\dot{b}_{\textup{\textsf{prod}}}$ be a $\bL$-name evaluating to the function
\[
b_{\textup{\textsf{prod}}} = \{ \la (m,\stem(p)(d^{-1}_{\rm min}(m))), p \ra \such p \in G, m \in \omega, \dom(p) > d^{-1}_{\rm min}(m)\}.
\]
\end{enumerate}\end{definition}

\begin{lemma}\label{cummings-magidor-repetitions} Let $p \in \bL$ with $|\stem(p)|=d^{-1}_\textup{min}(n)$ and suppose $\dot{\gamma}$ is a name for an ordinal that is forced by $p$ to be below $\dot{b}_\mathsf{prod}(n)$. Then there is some $q \le_0 p$ and some $\delta<\aleph_n$ such that $q \Vdash \textup{``}\dot{\gamma} \le \delta \textup{''}$.\footnote{In our discussions with Hannes Jakob, we realized that the previous version's proof should not apply to forcings with a higher splitting cofinality; he found a counterexample for the other cases and suggested that the proof should work for cofinality $\omega_1$ splitting.}\end{lemma}

\begin{proof} Let $\alpha \in \osucc_p(\stem p):=A$. It is sufficient to find some $q^A_\alpha \le_0 p \rest \stem(p) {}^\frown \langle \alpha \rangle$ and some $\delta_\alpha<\alpha$ for each $\alpha \in A$ such that $q^A_\alpha \Vdash \dot{\gamma} \le \delta_\alpha<\alpha$. Since $\alpha \mapsto \delta_\alpha$ is regressive, there is therefore a stationary subset $S \subseteq \osucc_p(\stem p)$ and some $\delta_*$ such that $\delta_\alpha = \delta_*$ for all $\alpha \in S$. Then we can let $q = \bigcup_{\alpha \in S}q^A_\alpha$ and we see that $q$ forces $\dot{\gamma} \le \delta_*$.

For the remainder of the proof, fix $\beta \in A$ and choose a cofinal sequence $\seq{\delta_i}{i<\omega_1}$ in $\beta$ (using the fact that the splitting sets concentrate on cofinality $\omega_1$). Note that $p \rest \stem p {}^\frown \langle \beta \rangle$ forces $\beta = \dot{b}_\mathsf{prod}(d^{-1}_\textup{min}(n))$. For each splitting node $t \in p \rest \stem(p) {}^\frown \langle \beta \rangle$, we say that $t$ is \emph{bad} if our sufficient statement also fails with respect to $p \rest t$, i.e$.$ if there is no $q \le_0 p \rest t$ such that for some $\delta< \aleph_n$, $q \Vdash \textup{``}\dot{\gamma} \le \delta\textup{''}$.

Now we will build an extension $r \le_0 p$ such that for all $t \in r$, $t$ is bad. This will give a contradiction because if $r' \le r$ is any condition deciding a bound for $\dot{\gamma}$ and $t = \stem(r')$, then $r' \le_0 \rest p \rest t$, contradicting badness of $t$.

We will build $r \le_0 p$ using a fusion sequence $\seq{p_k}{k<\omega}$ where $p_0 = p$ and where all $t \in p_k$ with $|t| \le |\stem(p)|+k$ are bad. This is fulfilled by $p_0$ by assumption. Suppose then that we have $p_k$. Let $t \in p_k$ be such that $|t|=|\stem (p)|+k$. 

\begin{claim} If $t$ is bad, then $S_t:= \{\alpha \in \osucc_t(p_n): t {}^\frown \la \alpha \ra \textup{ is bad}\}$ is a stationary subset of $\aleph_{d(|t|)} \cap \cof(\omega_1)$.\end{claim}

\begin{proof} Suppose otherwise. Then there is a stationary subset $S$ of $\alpha \in \osucc_t(p_n)$ such that $t {}^\frown \langle \alpha \rangle$ is not bad. Therefore, for each $\alpha \in S$, there is a $q_\alpha \le_0 p \rest t$ and some $\epsilon_\alpha<\beta$ such that $q_\alpha \Vdash \dot{\gamma} \le \epsilon_\alpha$. In fact, for each $\alpha$ we can rather choose $i_\alpha<\omega_1$ such that $q_\alpha \Vdash \dot{\gamma} \le \delta_{i_\alpha}$ where we are referring to the sequence of $\delta$'s fixed above. By completeness of the nonstationary ideal for $\aleph_n \cap \cof(\omega_1)$ and the fact that $n > 1$, there is a stationary subset $S' \subseteq S$ and some $i_*$ such that $i_\alpha = i_*$ for all $\alpha \in S'$. Then let $q = \bigcup_{\alpha \in S'}q_\alpha$. Then $q$ forces $\dot{\gamma} \le \delta_{i_*} < \beta$. This contradicts the assumption that $t$ is bad.\end{proof}

Since $S_t$ is a stationary subset of $\aleph_{d(|t|)}\cap \cof(\omega_1)$ for all such $t$, we let $p_{n+1} \cup \{p \rest (t {}^\frown \la \alpha \ra):\alpha \in S_t\}$. Having defined $p_n$ for $n<\omega$, we let $r = \bigcap_{n<\omega}p_n$. Then $r$ is the condition described above such that $t$ is bad for all $t \in r$, hence we have finished the proof.
\end{proof}

Note that the proof of \autoref{cummings-magidor-repetitions} uses the fact that $\bL$ is defined to split into stationary sets by invoking Fodor's Lemma. The same is true of the next lemma.

Now we are able to prove a bounding lemma analogous to one obtained by Cummings and Magidor \cite[Fact 4]{Cummings-Magidor2011}. This lemma embodies the reason that we are using Laver-style Namba forcings in this paper.

\begin{lemma}\label{bounding} Let $\vec{f}=\seq{f_\alpha}{\alpha<\aleph_{\omega+1}}$ be a scale on $\prod_{n<\omega}\aleph_n$. Then $\bL$ forces that $\dot{b}_{\textup{\textsf{prod}}}$ is an exact upper bound of $\vec{f}$.\end{lemma}

\begin{proof} First we need to argue that $\bL$ forces $\dot{b}_{\textup{\textsf{prod}}}$ to be an upper bound of $\vec{f}$. This follows from a relatively simple argument: If $\alpha<\aleph_{\omega+1}$ and $p \in \bL$, construct $q \le_0 p$ such that for all $k=d^{-1}_\textup{min}(n) \ge |\stem q|$ and $t \in q$ with $|t|=k$, $\osucc_q(t) \setminus f_\alpha(k) = \emptyset$.

Now we prove the more complicated assertion, which is that if $p \Vdash \textup{``}\dot{h} < \dot{b}_{\textup{\textsf{prod}}}\textup{''}$ where $\dot{h}$ is taken to be arbitrary, then we can find $q \le p$ such that for some $\xi<\aleph_{\omega+1}$, $q \Vdash \textup{``}\dot{h}<^* f_\xi\textup{''}$.

Let $\{k_n:n<\omega\}$ enumerate $\{d^{-1}_\textup{min}(n):n<\omega\}$. Let $N$ be such that $p \Vdash \textup{``}\forall n \ge N,\dot{h}(n) < \dot{b}_{\textup{\textsf{prod}}}(n)\textup{''}$.

We will define a fusion sequence $\seq{p_n}{n<\omega}$ and values $g(n)$ of a function $g$ by induction on $n$ such that $p_n \Vdash \textup{``}\dot{h}(n)<g(n)\textup{''}$ for $n \ge N$. Let $p_n = p$ for $n \le N$. Suppose we have defined $p_{n-1}$ and $n \ge N$. For all $t \in p_{n-1}$ with $|t|=k_n$ and all $\alpha \in \osucc_{p_{n-1}}(t)$, apply \autoref{cummings-magidor-repetitions} to find $q'_{t {}^\frown \la \alpha \ra} \le p_{n-1} \rest t$ and some $\delta_{t {}^\frown \la \alpha \ra}$ such that $q'_{t {}^\frown \la \alpha \ra} \Vdash \textup{``}\dot{h}(n) \le \delta_{t {}^\frown \la \alpha \ra}\textup{''}$. By Fodor's Lemma, there is a stationary $S_t \subseteq \osucc_{p_{n-1}}(t)$ and a value $\delta_t$ such that for all $\alpha \in S_t$, $q_{t {}^\frown \la \alpha \ra} \Vdash \textup{``}\dot{h}(n) \le \delta_t \textup{''}$. Let $q_t = \bigcup_{\alpha \in S_t}q'_{t {}^\frown \la \alpha \ra}$.

 Then let $p_n = \bigcup \{q_t:t \in p_{n-1},|t|=k_n\}$. Let $g(n) = \sup \{ \delta_t:t \in p_{n-1},|t|=k_n\}$. Since $d(n')<m$ for all $n' < k_n$, it follows that $|\{t \in p_n:|t|=k_n\}|<\aleph_n$, and so $g(n)<\aleph_n$.

 We finally let $q = \bigcap_{n<\omega}p_n$ be the fusion limit.  Observe that $q \Vdash \textup{``}\dot{h}(n) \le g(n)$ for $n \ge N$. If $\xi<\aleph_{\omega+1}$ is large enough that $g<^* f_\xi$ then we are done. \end{proof}

\subsection{A Poset for Adding a Good Scale}\label{goodscale-forcing-sec}

Here we will develop a poset that forces that there is a good scale.

Fix a singular $\lambda$ of cofinality $\kappa$ and let $\seq{\lambda_i}{i<\kappa}$ be a strictly increasing sequence of regular cardinals converging to $\lambda$. We say $f <_j g$ if for all $i \ge j$, $f(i)<g(i)$. Hence $f<^\ast g$ if $f <_j g$ for some $j<\kappa$.

We define a poset for forcing a good scale.

\begin{definition} Given some $\vec \lambda =\seq{\lambda_i}{i<\kappa}$, let $\mathbb{G}(\vec \lambda)$ be a partial order whose conditions have the form $\seq{f_\beta}{\beta \le \alpha}$ for some $\alpha<\lambda^+$ such that for all $\beta \le \alpha$:

\begin{enumerate}

\item $f_\beta \in \prod_{i<\kappa}\lambda_i$;
\item for all $\gamma<\beta$, $f_\gamma<^\ast f_\beta$;
\item if $\cf(\beta)>\kappa$, then $\beta$ is a good point with respect to $\seq{f_\gamma}{\gamma<\beta}$.

\end{enumerate}

Ordering is by end-extension: if $p,q \in \mathbb{G}(\vec \lambda)$, then $p \le q$ if and only if $p \rest \dom q = q$. We drop the notation for $\vec \lambda$ when the context is clear.\end{definition}

\begin{proposition}\label{gs-poset-directed} $\mathbb{G}(\vec \lambda)$ is $\kappa^+$-directed closed.\end{proposition}

\begin{proof} $\mathbb{G}(\vec \lambda)$ is \emph{tree-like}, meaning that $p,q \in \mathbb{G}$ are compatible if and only if $p \le q$ or $q \le p$. Therefore it is enough to show that $\mathbb{G}(\vec \lambda)$ is $\kappa^+$-closed. This follows from the facts that points $\beta$ with $\cf(\beta)<\kappa$ are automatically good and that we do not require points $\beta$ with $\cf(\beta) = \kappa$ to be good.\end{proof}

\begin{proposition} $\mathbb{G}(\vec \lambda)$ is $(\lambda+1)$-strategically closed.\footnote{For the definition of strategic closure, see \cite[Definition 5.15]{Handbook-Cummings}.}\end{proposition}

\begin{proof} The play will take the form of a decreasing sequence $\seq{p_\xi}{\xi \le \lambda} \subseteq \G(\vec \lambda)$ in which $\gamma_\xi = \max \dom p_\xi$, i.e$.$ each $p_\xi$ will formally have the presentation $p_\xi = \seq{f^\xi_\zeta}{\zeta \le \gamma_\xi}$, but because we have $p_\xi \le p_{\xi'}$ for $\xi > \xi'$, we can write $p_\xi = \seq{f_\zeta}{\zeta \le \gamma_\xi}$. Player $\textup{\textsf{II}}$ will play so that if $\xi < \xi' < \lambda$ are such that $\xi,\xi'$ are even and $j$ is minimal such that $\xi<\xi' < \lambda_j$, then $f_\xi <_j f_{\xi'}$.

Suppose $\xi<\lambda$ is an even successor with $j$ minimal such that $\xi<\lambda_j$ and $\xi = \eta+2$. Then Player $\textup{\textsf{II}}$ will choose $h$ such that $p_{\eta+1}(\gamma_{\eta+1}) <^* h$ and $p_\eta(\gamma_\eta) <_j h$ and will play $p_\xi:= p_{\eta+1} {}^\frown \langle \gamma_{\eta+1}+1,h \rangle$.

Suppose $\xi<\lambda$ is a limit and $j$ is minimal such that $\xi<\lambda_j$. Then let $\gamma_\xi = \sup_{\eta<\xi}\gamma_\eta$ and let $h(i) = \sup \{f_{\gamma_\eta}(i): \eta<\xi,\eta \textup{ even}\}$ for $i \ge j$ and $h(i) = 0$ otherwise. If $\cf(\xi) \le \kappa$, there is no consideration with regard to goodness. If $\cf(\xi) > \kappa$, then if we let $f^*_{\gamma_\eta}(i)=f_{\gamma_\eta}(i)$ for $i \ge \lambda_j$ and $f^*_{\gamma_\eta}(i)= 0$ otherwise, then it follows by construction that $\seq{f^*_{\gamma_\eta}}{\eta<\xi,\eta \textup{ even}}$ is $<_j$-increasing and cofinally interleaved with $\seq{f_{\gamma_\eta}}{\eta<\xi}$. This is one of the equivalent definitions of goodness, so if we define $p_\xi$ such that $\dom p_\xi = \gamma_\xi+1$, $p_\xi \le p_\eta$ for $\eta<\xi$, and $p_\xi(\gamma_\xi) = h$ where $h$ is an exact upper bound of $\seq{f_{\gamma_\eta}}{\eta<\xi}$, then $p_\xi$ is a condition.

If $\xi=\lambda$, then we can find a lower bound by \autoref{gs-poset-directed}.\end{proof}

By distributivity we have cardinal preservation.

\begin{proposition} If $2^\lambda=\lambda^+$ then $\mathbb{G}(\vec \lambda)$ preserves cardinals and cofinalities.\end{proposition}

\begin{proposition}\label{g-adds-good-scale} $\mathbb{G}(\vec \lambda)$ adds a good scale to $\vec \lambda$.\end{proposition}

\begin{proof} If $\dot{h}$ is a $\mathbb{G}(\vec \lambda)$-name for a function in the product as forced by some condition $p$, then choose $p' \le p$ such that $p' \Vdash \textup{``}\dot{h} = g\textup{''}$. Then choose $p'' \le p'$ such that $p''(\max\dom p'')$ dominates $g$.\end{proof}

\section{Proving the Main Theorem}

This section constitutes the proof of \autoref{maintheorem}.

\subsection{Defining the Iteration and Establishing Basic Properties of the Target Model}\label{section-setup}

Now we define the model that witnesses the main theorem.

First we establish some notation. For this section, and some model $V_0$, let $\mathbb{L}^{V_0}$ refer $\bL$ is defined in the model $V_0$.

If $\tau$ is a cardinal, we recall Jensen's forcing $\S_\tau$ notion for adding a $\square_\tau$-sequence: Conditions are functions $s$ such that:

\begin{enumerate}

\item $\dom s \in \tau^+$,

\item $\forall \alpha \in \dom s$, $s(\alpha)$ is a closed unbounded subset of $\alpha$, of order-type $\le \tau$,

\item $\forall \alpha,\beta \in \dom s$, if $\beta$ is a limit point of $s(\alpha)$, then $s(\alpha) \cap \beta = s(\beta)$.

\end{enumerate}

We of course want:

\begin{fact} (See \cite[Section 6]{Cummings-Foreman-Magidor2001}.) The forcing $\S_\tau$:

\begin{enumerate}
\item is $(\tau+1)$-strategically closed and
\item adds a $\square_\tau$-sequence.
\end{enumerate}
\end{fact}

Now let $\S = \prod_{n<\omega}\S_{\aleph_n}$. Also, let $\mathbb{G}^{V_0} = \mathbb{G}(\prod_{n<\omega} \aleph_n^{V_0})$, the poset defined in \autoref{goodscale-forcing-sec}.

We make a standard definition explicit for clarity:

\begin{definition} A \emph{wide Aronszajn tree} of height $\aleph_1$ is a tree $T$ of height $\aleph_1$ and \emph{any} width that has no cofinal branches.\end{definition}

So, wide Aronszajn trees are not required to have countable levels.

We will also employ Baumgartner's specializing forcing:

\begin{fact}[Baumgartner] Suppose that $T$ is a wide Aronszajn tree of height $\aleph_1$. Then there is a forcing $\B(T)$ such that $\B(T)$ has the countable chain condition and such that $\B(T)$ adds a function $b:T \to \omega$ such that if $t \sqsubseteq t'$ are elements of $T$, then $b(t) = b(t')$ implies $t = t'$ (i.e$.$ $\B(T)$ forces $T$ to be special). (See \cite[Chapter 16]{Jech2003}.)\end{fact}

We need to describe a tree that will be used in this construction: Let $g$ be $\S \ast \dot{\G}$-generic and let $k$ be $ \dot{\mathbb{\mathbb{L}}}[g]$-generic over $V[g]$. Let $\nu$ be as in the statement of \autoref{mini-approx-thm-laver}. Let $D \subseteq \nu$ be a club of order-type $\omega_1$ in $V[g][k]$. Let $\mathcal{X}$ be the set of $\subset$-increasing and continuous chains in $H(\nu)^{V[g]} \cap ((H(\nu)^{V[g]})^{<\omega_1})$. Then let
\[
T_\textup{IA} = \{\seq{Z_i}{i<j} \in \mathcal{X}:\sup \bigcup_{i<j}(Z_i \cap \nu) \in D\}
\]
where the tree order is determined by end extension, i.e$.$ $\seq{Z^0_i}{i<j_0} \le_{T_\textup{IA}} \seq{Z^1_i}{i<j_1}$ if and only if $j_1 \ge j_0$ and $Z^0_i = Z^1_i$ for all $i<j_0$.

\begin{claim}\label{baumgartner-claim} Let $g$ be $\S \ast \dot{\G}$-generic, let $k$ be $\dot{\mathbb{\mathbb{L}}}[g]$-generic over $V[g]$. Let $\dot{\F}$ be the $\S \ast \dot{\G} \ast \dot{\bL}$ name for $\B(T_\textup{IA})$ if $T_\textup{IA}$ has height $\ge \omega_1$ and let $\dot{\F}$ be the name for the trivial forcing if $T_{\textup{IA}}$ has height $<\omega_1$. Let $f$ be $\dot{\F}$-generic over $V[g][k]$. Then if $W \supseteq V[g][k][f]$ is an outer model preserving $\omega_1$ (and hence $\cf(\nu)=\omega_1$), then there is no continuous sequence $\seq{M_i}{i<\omega_1}$ such that $\seq{M_i}{i<j} \in H(\nu)^{V[g]}$ for all $j<\omega_1$ and $H(\nu)^{V[g]} = \bigcup_{i<\omega_1}M_i$.\end{claim}

There is also a specific case of a theorem of Cummings that we will use for clarity \cite[Theorem 2]{Cummings1997}:

\begin{fact}[Cummings]\label{cummings-collapse} Let $W \supseteq V$ be an extension such that $W \models \textup{``}|\aleph_n^V|=\aleph_1\textup{''}$ for all $n<\omega$. if $V \models \textup{``There is a good scale on }\aleph_\omega\textup{''}$, then $W \models \textup{``}|\aleph_{\omega+1}^V|=\aleph_1 \textup{''}$.\end{fact}

\begin{proof}[Proof of \autoref{baumgartner-claim}] First we make some observations. Consider the tree $T$ consisting of elements of the form
\[
\seq{Z_i}{i<j} \in H(\nu)^{V[g]} \cap (H(\nu)^{V[g]})^{<\omega_1}
\]
defined without the restriction to $D$. By \autoref{mini-approx-thm-laver}, we know that $V[g][k] \models \textup{``}\cf(\nu)=|\nu|=|({}^{<\omega_1}\nu)^{V[g]}| = \aleph_1\textup{''}$, so $T$ has cardinality $\aleph_1$ in $V[g][k]$. Also by \autoref{mini-approx-thm-laver}, it follows that under the direct extension ordering, there are no branches $b$ of $T$ such that for all $\delta<\nu$, there is some $\seq{M_i}{i<j} \in b$ with $\bigcup_{i<j}\sup(M \cap \nu) \ge \delta$. It is immediate from $\ot(D)=\omega_1$ that $T_\textup{IA}$ has height no greater than $\omega_1$.

Now we consider two cases. The first is that the height of $T_{\textup{IA}}$ is equal to some $\gamma<\omega_1$ in $V[g][k]$. Suppose for contradiction that $\seq{M_i}{i<\omega_1}$ is a sequence as in the statement of the claim. Let $E:=\seq{\beta_i}{i<\omega_1}$ enumerate $\{\bigcup_{j<i}\sup(M_j \cap \nu):i<\omega_1\}$. By continuity, $E$ is a club, so $D \cap E$ is a club in $\nu$. Choose some $\delta \in D \cap E$ such that $\ot(D \cap E \cap \delta) > \gamma$. If $\delta = \beta_i$, and $j^*$ is such that $\bigcup_{i<j^*}(M_i \cap \nu) \in D$. Then $\seq{M_i}{i<j^*}$ has at least $\delta$-many predecessors in $T_{\textup{IA}}$, which is a contradiction. 

Now suppose that $T_{\textup{IA}}$ has height $\omega_1$ in $V[g][k]$. We argue that in $V[g][k]$, $T_\textup{IA}$ is a wide Aronszajn tree of cardinality and height $\omega_1$, in particular that $T_{\textup{IA}}$ has no cofinal branches in $V[g][k]$. By the observation stated in the first paragraph about branches that are unbounded in $\nu$, it is sufficient to consider the possibility of a branch $b$ of height $\omega_1$ such that
\[
\exists \beta<\nu, \forall j<\omega_1,\seq{M_i}{i<j} \in b \then \bigcup_{i<j}(M_i \cap \nu) < \beta.
\]
Let $\seq{\alpha_i}{i<\omega_1}$ be an increasing enumeration of $D$ and let $i^*$ be minimal such that $\alpha_{i^*} \ge \beta$. But for all $\seq{M_i}{i<j},\seq{M_i}{i<j+1} \in b$, there must be an element in $D$ which is in $M_{j+1} \setminus M_j$. Therefore this is impossible since there are countably many elements of $D$ below $\alpha_{i^*}$.

To finish the claim, suppose for contradiction $\seq{M_i}{i<\omega_1}$ is a sequence as in the statement. Let $E$ be as defined in the first case. Let $\seq{\gamma_i}{i<\omega_1}$ enumerate $D \cap E$ and let $\vec{M}_i=\seq{M_j}{j<\xi(i)}$ be the corresponding elements of $T_{\textup{IA}}$. Then each $\vec{M}_i$ belongs to a level of height $\ge i$ within $T_{\textup{IA}}$. This contradicts the facts $\omega_1$ is preserved and that the generic function added by $f$ can only take countably many values.\end{proof}

We start in a ground model $V$ in which $\kappa$ is a $\kappa^{+\omega+1}$-supercompact cardinal. The preparation is defined as follows: Fix a Laver supercompact guessing function $\ell:\kappa \to V_\kappa$ such that for every $x$ and $\lambda \ge |\tc(x)|$ up to $\kappa^{+\omega+1}$, there is a $\lambda$-supercompact embedding $j:V \to M$ with critical point $\kappa_0$ such that $j(\ell)(\kappa) = x$ \cite{Laver1978}.

We define a revised countable support (see \cite{Foreman-Magidor-Shelah1989}) iteration $\mathbb{I} = \seq{\mathbb{I}_\alpha,\dot{\mathbb{J}}_\alpha}{\alpha<\kappa}$ as follows:

\begin{enumerate}

\item Suppose $\alpha$ is inaccessible and that $\ell(\alpha)$ is an $\mathbb{I}_\alpha$-name for a poset of the form
\[
\dot{\S} \ast \dot{\G} \ast \dot{\mathbb L} \ast \dot{\B}(T_\textup{IA}) \ast \dot{\Col}(\aleph_1,\chi)
\]
where $T_\textup{IA}$ indicates the wide Aronszajn tree discussed above and $\chi$ is some regular cardinal. If $T_\textup{IA}$ has height $\omega_1$, then let $\dot{\mathbb{J}}_\alpha = \ell(\alpha)$. If $T_\textup{IA}$ has height less than $\omega_1$, let $\dot{\mathbb{J}}_\alpha= \dot{\S} \ast \dot{\G} \ast \dot{\mathbb L} \ast \dot{\Col}(\aleph_1,\chi)$.

\item Suppose $\alpha$ is inaccessible and $\ell(\alpha)$ is an $\mathbb{I}_\alpha$-name for a poset of the form
\[
\S \ast \dot{\G} \ast \dot{\Add}(\alpha^{+\omega+1}) \ast  \dot{\mathbb{L}}.
\]
 Then let $\dot{\mathbb{J}}_\alpha = \ell(\alpha)$.

\item If $\alpha$ is inaccessible and $\ell(\alpha)$ is an $\mathbb{I}_\alpha$-name for $\Col(\aleph_1,\chi)$ for some $\chi<\kappa$, then let $\dot{\mathbb{J}}_\alpha$ be a name for $\Col(\aleph_1,\chi)$.

\item Otherwise $\dot{\mathbb{J}}_\alpha$ is a name for the trivial poset.

\end{enumerate}

\begin{proposition} If $G$ is $\mathbb{I}$-generic over $V$, then $V[G] \models \textup{``}\kappa = \aleph_2$ and all posets of cardinality $\le \kappa^{+\omega+1}$ that preserve stationary subsets of $\omega_1$ are semiproper$\textup{''}$.\end{proposition}

\begin{proof} By standard arguments, $\mathcal{I}$ has the $\kappa$-chain condition and thus preserves regularity of $\kappa$. Because of the iterands for Case (2), there are surjections from $\aleph_1$ to $\chi$ for all $\chi<\kappa$, hence $\kappa$ cannot be any larger than $\aleph_2$ in the extension. The collapse iterands in (1) ensure the statement about semiproperness by a lemma of Foreman, Magidor, and Shelah \cite[Lemma 3]{Foreman-Magidor-Shelah1989}.\end{proof}

Let $G_{\mathbb{I}}$ be $\mathbb{I}$-generic over $V$, let $G_\S$ be $\S$-generic over $V[G_{\mathbb{I}}]$, and let $G_\G$ be $\G$-generic over $V[G_{\mathbb{I}}][G_\S]$. Then $V[G]:=V[G_\mathbb{I}][G_\S][G_\G]$ will be the model witnessing \autoref{maintheorem}.

\begin{proposition} The following are true in $V[G]$:

\begin{enumerate}
\item $\forall \tau > \aleph_1$, $2^\tau = \tau^+$,
\item For all $n<\omega$, $\square_{\aleph_n}$ holds,
\item All scales on $\aleph_\omega$ are good.
\end{enumerate}
\end{proposition}

\begin{proof} Without loss of generality, we can assume that $\GCH$ holds above $\kappa$. Then this is preserved by $\mathbb{I}$ since it has cardinality $\kappa$. \autoref{g-adds-good-scale} gives us the third point because if there is a good scale on the full product $\prod_{n<\omega}\aleph_n$, then restrictions to sub-products are automatically good.\end{proof}

\subsection{The Lifting Argument}

Most of the work here consists of the following:

\begin{lemma} In $V[G]$, there are stationarily-many $N \prec H(\aleph_{\omega+1})$ of cardinality $\aleph_1$ that are not sup-internally approachable.\end{lemma}

We will in fact phrase the result in terms of guessing models.

\begin{definition} Let $N$ be a set such that $\aleph_1 \subseteq N$. Assume $\tau$ and $\lambda$ are regular uncountable cardinals. Then we say that $N$ is \emph{weakly $(\tau,\lambda)$-guessing} if whenever $f:\tau \to \ON$ and $f \rest i \in N$ for cofinally many $i < \sup(N \cap \tau)$ and $f$ is unbounded in $\sup(N \cap \lambda)$, then there is some $g \in N$ such that $f \rest N = g$.\end{definition}

Observe that if $N$ is $(\omega_1,\lambda)$-weakly guessing for some $\lambda$, then we can take $f \in N$. Also, if $N$ is sufficiently elementary, the definition implies that there can be no such $f$.

\begin{proposition} Suppose $\lambda>\aleph_1$ and that $N$ with $\aleph_1=|N|=\cf(\aleph_1)$ is weakly $(\omega_1,\lambda)$-guessing. Then $N$ is not sup-internally approachable at $\lambda$.\end{proposition}

\begin{proof} Suppose that $N$ is sup-internally approachable via $\seq{M_i}{i<\omega_1}$. Let $\delta = \sup(N \cap \lambda)$. Let $A = \seq{\sup(M_i \cap \lambda)}{i<\omega_1}$. Then $N$ has all initial segments of $A$. Since $A$ is unbounded in $\delta$, it follows that $N$ cannot be weakly $(\omega_1,\kappa)$-guessing.\end{proof}

\begin{lemma}\label{guessing-lemma} In $V[G]$, for all $\theta \ge \aleph_{\omega+1}$, there are stationarily-many $N \prec H(\theta)$ of cardinality $\aleph_1$ such that:
\begin{enumerate}
\item $\cf(N \cap \aleph_{\omega+1})=\aleph_1$,
\item $N$ is weakly $(\omega_1,\aleph_{\omega+1})$-guessing.
\end{enumerate}\end{lemma}

\begin{proof}[Proof of \autoref{guessing-lemma}]

Let $\nu = \kappa^{+\omega+1} =  \aleph_{\omega+1}^{V[G]}$.

Use the properties of the Laver function to find an embedding $j:V \to M$ such that $j(\kappa)>\nu$, $M^{\nu} \subset M$, and $j(\ell)(\kappa)$ is the $\mathbb{I}$-name for
\[
\S \ast \dot{\G} \ast \dot{\mathbb{L}}  \ast  \dot{\B}(T_\textup{IA}) \ast \dot{\Col}(\aleph_1,\chi).
\]
We let $\rho:=\sup j[\nu]$.

First we will show that we can obtain a lift of $j:V \to M$ in a well-behaved forcing extension:

\begin{claim}\label{lifting-claim} There is a forcing extension $V[G][K] \supset V[G]$ such that:

\begin{enumerate}

\item  $j:V \to M$ can be lifted to $j:V[G] \to M[j(G)]=M[G][K]$,

\item $V[G][K] \models \textup{``}|\aleph_{\omega+1}^{V[G]}|=|\aleph_1^{V[G]}|=\aleph_1\textup{''}$,

\item In $M[G][K]$, if $f:\omega_1 \to \ON$ is not bounded in $\nu$, then there is some $i<\omega_1$ such that $f \rest i \notin V[G]$.

\end{enumerate}\end{claim}

This gives us the material to obtain the models:

\begin{claim}\label{contradiction-claim} If $N_0 = H(\nu)^{V[G]}$, then in $M[G][K]$, the following hold:
\begin{enumerate}
\item $|j[N_0]| = \aleph_1 \subseteq j[N_0]$,
\item $j[N_0] \prec j(N_0)$,
\item $\cf(j[N_0] \cap j(\nu)) = \omega_1$,
\item $j[N_0]$ is weakly $(\omega_1,\aleph_{\omega+1})$-guessing.
\end{enumerate}
\end{claim}

Now we can prove the claims.

\begin{proof}[Proof of \autoref{lifting-claim}] Let $j$ be as fixed at the beginning of the proof of \autoref{guessing-lemma}. We will lift $j$ in a series of steps in which we extend the range of $j$ to the next generic. We will establish conditions \emph{(1)} and \emph{(2)} in the first step and argue that it is preserved as long as $\aleph_1$ is preserved in the following steps.

\emph{Lifting to domain $V[G_{\mathbb{I}}]$:} By elementarity and Laver guessing and the fact that the iterands preserve stationary subsets of $\omega_1$ (\autoref{namba-omega1-pres} for the case of $\bL$) we have that
\[
\dot{\S} \ast \dot{\G} \ast \dot{\mathbb{L}} \ast \dot{\B}(T_\textup{IA}) \ast \dot{\Col}(\aleph_1,\chi)  
\]
is forced by $\mathbb{I}$ to preserve stationary subsets of $\omega_1$.

 Therefore, by the definition of the iteration and the output of the guessing function, we have
\[
j(\mathbb{I}) = \mathbb{I} \ast \dot{\S} \ast \dot{\G} \ast \dot{\mathbb{L}}  \ast \dot{\B}(T_\textup{IA}) \ast \dot{\E}
\]
where $\dot{\E}$ is a remainder term that includes the L{\'e}vy collapse component. (This is where we use $\kappa^{+\omega+1}=\nu$-supercompactness of the embedding.)

Let $K_\mathbb{L}$ be $\mathbb{L}$-generic over $V[G]$, let $K_{\mathbb{B}}$ be $\B(T_\textup{IA})$-generic over $V[G][K_\mathbb{\bL}]$,and let $K_{\E}$ be $\E$-generic over $V[G][K_{\mathbb{L}}][K_{\mathbb{B}}]$.

Then since $j" \mathbb{I} \subseteq \mathbb{I}$, Silver's classical lifting argument (see \cite{Handbook-Cummings}) gives us an embedding $j:V[G_{\mathbb{I}}] \to M[G_{\mathbb{I}} \ast G_\S \ast G_\G \ast K_{\mathbb{L}} \ast K_{\B} \ast K_{\E}] = M[j(G_{\mathbb{I}})]$.

The argument that Condition \emph{(2)} holds works as follows: The L{\'e}vy collapse iterand collapses $\nu$ to have cardinality and cofinality $\aleph_1$. Hence it is sufficient for Condition \emph{(2)} to demonstrate preservation of $\aleph_1$ in the next steps.

In $V[G][K_{\mathbb{L}}][K_{\mathbb{B}}]$, $T_\textup{IA}$ is special, and in particular still a wide $\aleph_1$-Aronszajn tree. We can assume that its height is $\omega_1$ since otherwise \autoref{baumgartner-claim} indicates a trivial case. By \autoref{baumgartner-claim}, a function with the properties included in Condition \emph{(3)} can be used to construct a chain that makes up a cofinal branch of $T_\textup{IA}$ (see e.g. \cite[Theorem 3.11]{Handbook-Eisworth}). Such a chain would collapse $\aleph_1$. Therefore, to ensure Condition \emph{(3)} it is enough to show that $\aleph_1$ is preserved as computed in terms of extensions of $M$.

\emph{Lifting to domain $V[G_{\mathbb{I}}][G_{\S}]$}: Let $\gamma_n:=\sup j[\aleph_n^{V[G_\mathbb{I}]}]$ for all $n<\omega$. Then $\cf(\gamma_n) = \omega$ in $V[j(G_{\mathbb{I}})]$. Now consider $\bigcup_{s \in G_\S}j(s)$. Taking a cofinal $\omega$-sequence $r_n \subseteq \gamma_n$ for each $n<\omega$, if we take $s^*(n) = \bigcup_{p \in G_\S}j(s(n))$, then we can inductively argue that $\bar{s}:=\seq{s^*(n) \cup \la \gamma_n,r_n \ra}{n<\omega}$ is a master condition for $j[G_\S]$ in $j(G_\S)$ since coherence for the condition at $\gamma_n$ is trivial because there are no limit points. (See \cite{Cummings-Foreman-Magidor2003} for more detail.)

Then $\bar{s} \in j(\S)$ is a master condition for $j[G_{\S}]$. Then we choose $K_{\S}$ to be a $j(\S)$-generic containing $p$.

\emph{Lifting to domain $V[G_{\mathbb{I}}][G_\S][G_{\G}]$:} We use another master condition argument. Let $\rho = \sup j[ \mu]$.  Let $\bar p = \bigcup_{p \in G_\G}j[p] = j[\vec f]$. Let $j(\vec f) = \seq{f^*_\alpha}{\alpha<j(\mu)}$.

We now argue that $\bar{p}$ can be extended to a condition in $j(\G)$. It is sufficient to argue that $M[j(G_{\mathbb{I}}) \ast j(G_{\mathbb{\S}})]$ models that $\bar{p}$ can be extended to a condition in $j(\G)$ since $j(\G) \in M[j(G_{\mathbb{I}}) \ast j(G_{\mathbb{\S}})]$. It is enough to argue from the perspective of $M[j(G_{\mathbb{I}} \ast G_{\S})]$. Specifically, we need to argue that $\rho$ is a good point of $j(\vec f)$, which then implies that $\bar{p} {}^\frown \langle \rho,h \rangle$ is a condition where $h$ is any exact upper bound of $\bar{p}$. Let $b_{\textup{\textsf{prod}}}$ be the restriction of the generic function added by $\bL$ denoted using the notation from \autoref{generic_min_function}. Enumerate $b_{\textup{\textsf{prod}}}$ as $\seq{b_n}{n<\omega}$. Abusing notation slightly, let $j(b_{\textup{\textsf{prod}}})$ denote the function $n \mapsto j(b_n)$. Note that even though $b_{\textup{\textsf{prod}}}$ is not in the domain of $j$, $M[j(G_{\mathbb{I}} \ast G_{\S})]$ contains $j \rest \kappa^{+n}$ for all $n<\omega$, and $b_{\textup{\textsf{prod}}} \in M[j(G_{\mathbb{I}} \ast G_{\S})]$, so $j(b_{\textup{\textsf{prod}}}) \in M[j(G_{\mathbb{I}} \ast G_{\S})]$. We will argue that the function $j(b_{\textup{\textsf{prod}}})$ is an exact upper bound of $\bar{p}$. This will be sufficient since the range of this function consists  points of cofinality $\aleph_1$ by the third point of \autoref{laver-namba-def-rep}. Therefore we can apply the definition of goodness from \autoref{goodness-characterized}.

\begin{claim}\label{semiproper-bounding-claim} Suppose that $h \in V[j(G_{\mathbb{I}} \ast G_{\S})]$ and that $h <^* j(b_{\textup{\textsf{prod}}})$. Then there is some $\alpha<\nu$ such that $M[j(G_{\mathbb{I}} \ast G_{\S})] \models \textup{``}h <^* f^*_{j(\alpha)}\textup{''}$ .
\end{claim}

\begin{proof} We use the fact that $\cf(b_{\textup{\textsf{prod}}}(n))=\aleph_1$ in $V[G_{\mathbb{I}} \ast G_{\S}]$. Fixing any $n<\omega$:, the image of $j$ is unbounded in $\cf(j(b_{\textup{\textsf{prod}}})(n))$, so we can assume that $h(n)$ is in the image of $j$ for all $n<\omega$. Let $h'$ be the function such that $j(h'(n)) = h(n)$ for all $n<\omega$. We have that $j(G_{\mathbb{I}} \ast G_{\S}) = G \ast K_{\mathbb{L}} \ast K'$ where $K'$ is a generic for a semiproper forcing in $V[G \ast K_{\bL}]$. In particular, semiproper forcings are strictly bounding for functions $\omega \to \omega_1$. The space $\{f:f<b_{\textup{\textsf{prod}}}\}$ is cofinally interleaved with an embedding of ${}^\omega \omega_1$. Therefore there is some $h'' \in V[G][K_{\mathbb{L}}]$ such that $h' < h''$. By \autoref{bounding}, there is some $\alpha < \aleph_{\omega+1}^{V[G]}$ such that $h'' <^* f_\alpha$. Suppose this is witnessed by $m<\omega$. Then for all $n \ge m$, we have that $j(f_\alpha(n)) = f^*_{j(\alpha)}(n) > j(h'(n)) = h(n)$. Hence $M[j(G_{\mathbb{I}} \ast G_{\S})] \models \textup{``}h <^* f^*_{j(\alpha)}\textup{''}$.\end{proof}

Applying \autoref{semiproper-bounding-claim}, we have that where $j(\alpha)<\rho$. Hence $\bar p {}^\frown \langle \rho,h \rangle$ is a master condition for $j"G_\G$. Let $K_\G$ be any generic for $j(\G)$ containing $\bar p$. Since $V[G_{\mathbb{I}}][G_{\S}] \models \textup{``}\G$ is countably closed$\textup{''}$, $M[j(G_{\mathbb{I}}][j(G_{\S})] \models \textup{``}j(\G)$ is countably closed$\textup{''}$ and therefore preserves $\aleph_1$, this completes the lifting argument.

This completes the steps of the lifting argument. Finally, we let $K = K_{\mathbb{L}} \ast K_{\mathbb{B}}  \ast K_{\E}  \ast K_\S \ast K_\G $.\end{proof}

Before starting with the proof of \autoref{contradiction-claim}, we establish some claims about the main object of interest generated by the lift. This effort allows us to speak of elementary submodels of $H(\aleph_{\omega+1})$ in particular.

\begin{lemma}\label{easton-htheta} Suppose that $\P \subseteq H(\aleph_{\omega+1})$, that $\P$ is countably closed, and that for arbtirarily high $n<\omega$, $\P \cong \P_A \times \P_B$ where $\P_A$ has cardinality $\aleph_n$ and $\P_B$ is $(\aleph_n+1)$-strategically closed. Then if $g$ is $\P$-generic, then $H(\aleph_{\omega+1})[g] = H(\aleph_{\omega+1})^{V[g]}$.\end{lemma}

\begin{proof} Note that if $\aleph_{\omega+1}$ is not preserved by $\P$, then $V[g]\models \textup{``}\cf(\aleph_{\omega+1}^V)=\aleph_n\textup{''}$. The factorization property of $\P$ therefore implies, using a variant of Easton's Lemma (see \cite[Remark 5.17]{Handbook-Cummings}), that $\aleph_{\omega+1}$ is preserved.

First we argue that $H(\aleph_{\omega+1})[g] \subseteq H(\aleph_{\omega+1})^{V[g]}$ by $\in$-induction. Suppose that $\dot{X} \in H(\lambda)$ is a $\P$-name. Since any element of $\dot{X}$ is forced to be equal to some $\dot{z}$ such that $\pair{\dot{z}}{p} \in \dot{X}$, and the $\dot{z}$'s are by induction forced to have transitive closure of cardinality $\le \aleph_\omega$, the rest follows by $\in$-induction.

Now we argue that $H(\aleph_{\omega+1})^{V[g]} \subseteq H(\aleph_{\omega+1})[g]$. Again we can argue by $\in$-induction. Suppose that $\dot{X}$ is forced to have transitive closure of cardinality strictly less than $\aleph_{\omega+1}$. Thus, we suppose that $r \in \P$ and $r \Vdash \textup{``}\seq{\dot{y}_\xi}{\xi<\tau} \textup{ enumerates }\dot{X}\textup{''}$ for some $\tau \le \aleph_\omega$ and moreover $r \Vdash \textup{``}\forall \xi<\tau,\exists \dot{w} \in H(\aleph_{\omega+1})^V,\dot{y}_\xi = \dot{w}\textup{''}$. Then it is sufficient to argue that there is some $r' \le r$ and some $\seq{\dot{z}_\xi}{\xi<\tau} \subseteq H(\aleph_{\omega+1})^V$ such that $r' \Vdash \textup{``}\dot{X} = \seq{\dot{z}_\xi}{\xi<\tau}\textup{''}$.

First suppose that $\tau<\aleph_\omega$. Let $\P \cong \P_A \times \P_B$ be such that $\P_A$ has cardinality $\aleph_n$ and $\P_B$ is $(\aleph_n+1)$-strategically closed where $\aleph_n > \tau$. Let $r$ be identified with $(\bar{p},\bar{q})$ under the factorization. (We are suppressing the details of the isomorphism $\P \cong \P_A \times \P_B$.) For each $\xi<\tau$, let $D_\xi \subseteq \P_B$ be the set of conditions $q$ such that there is a maximal antichain $A \subseteq \P_A$ below $\bar{p}$ and a matrix $\{\dot{z}^\xi_{p,q}:p \in A\}$ such that $\dot{z}^\xi_{p,q} \in H(\aleph_{\omega+1})$ and $(p,q) \Vdash \textup{``}\dot{y}_\xi = \dot{z}^\xi_{p,q} \textup{''}$. We argue presently that the $D_\xi$'s are dense below $\bar{q}$: Suppose $q' \le \bar{q}$. Use the $(\aleph_n+1)$-strategic closure to build a sequence of conditions $\seq{q_i}{i \le \aleph_n}$ below $q'$ so that for each $i<\aleph_n$, there is some $p_i \in \P_A$ and $\dot{z}^\xi_{p_i,q_i} \in H(\aleph_{\omega+1})^V$ such that $(p_i,q_i) \Vdash \textup{``}\dot{y}_\xi = \dot{z}^\xi_{p_i,q_i}\textup{''}$, and do so in such a way that the $p_i$'s are incompatible. Since $|\P_A| = \aleph_n$, this leads to the construction of a maximal antichain. Then $q_{\aleph_n} \in D_\xi$.

Having argued for the density of the $D_\xi$'s, and noting that they are therefore open dense, use the strategic closure of $\P_B$ to find some $q^*$ in the intersection $\bigcap_{\xi<\tau}D_\xi$. For each $\xi<\tau$, let $A_\xi$ and $\{\dot{z}^\xi_{p,q^*}:p \in A_\xi\}$ witness that $q^* \in D_\xi$. Then let $\dot{z}_\xi$ be the name that glues the $\dot{z}^\xi_{p,q^*}$'s together along the antichain $A_\xi$. Since $\P \subseteq H(\aleph_{\omega+1})$, it follows that $\dot{z}_\xi \in H(\aleph_{\omega+1})$. Then $(\bar{p},q^*)$ forces that $\dot{X}$ is enumerated by $\seq{\dot{z}_\xi}{\xi<\tau}$, which is an element of $H(\aleph_{\omega+1})$.

Now suppose that $\tau = \aleph_\omega$. Let $\seq{\dot{y}_\xi}{\xi<\aleph_\omega}$ be forced to be an enumeration of $\dot{X}$. Then by the case for $\tau<\aleph_\omega$, we can build a $\le_\P$ decreasing sequence $\seq{p_n}{n<\omega}$ in $\P$ such that $p_n$ forces $\seq{\dot{y}_\xi}{\aleph_{n-1} \le \xi<\aleph_n} = \seq{\dot{z}_\xi}{\aleph_{n-1} \le \xi<\aleph_n}$ (setting $\aleph_{-1}=0$) where the $\dot{z}_\xi$'s are elements of $H(\aleph_{\omega+1})$. Then the lower bound of the $p_n$'s forces that $\dot{X}$ is enumerated by $\seq{\dot{z}_\xi}{\xi<\aleph_\omega}$, which is an element of $H(\aleph_{\omega+1})$.\end{proof}

\begin{claim}\label{htheta-claim} We have
\[
H(\nu)[G]:=\{\dot{a}^G:\dot{a} \in H(\nu)\} = H(\nu)^{V[G]}.
\]
\end{claim}

\begin{proof} It is a classical fact (the proof of which is roughly contained in the proof of \autoref{easton-htheta}) that if $\P$ is $\lambda$-cc for regular $\lambda$ and $\P \subseteq H(\lambda)$, then $H(\lambda)[G] = H(\lambda)^{V[G]}$. It therefore follows that $H(\nu)[G_{\mathbb{I}}]=H(\nu)^{V[G_{\mathbb{I}}]}$. \autoref{easton-htheta} implies that if $H'=H(\nu)^V[G_{\mathbb{I}}]$, then $H'[G_{\S}]=(H')^{V[G_{\mathbb{I}}][G_{\S}]}$. Finally, observe that if $H''=H(\nu)^{V[G_{\mathbb{I}}][G_{\S}]}$, then we easily have $H''[G_{\G}]=(H'')^{V[G]}$ from the fact that $\dot{\G}$ is forced to be $\aleph_{\omega+1}$-distributive.\end{proof}

\begin{claim}\label{containment-claim} If $j:V[G] \to M[G][K]$ is defined as in \autoref{lifting-claim}, then $j[N_0] \in M[G][K]$.\end{claim}

\begin{proof} By \autoref{htheta-claim}, we have that $H(\nu)^V[G] = H(\nu)^{V[G]}$. Since $M^{\nu} \subset M$, we have both $H(\nu)^V \subseteq M$ and $j[H(\nu)^V] \in M$. Since $G \ast K \in M[G \ast K]$, it follows that
\[
j[N_0] =  j[H(\nu)^V[G]] = \{j(\dot{a}^{G}):\dot{a} \in H(\nu)^V \} = \{j(\dot{a})^{G \ast K}:\dot{a} \in H(\nu)^V\},
\]
is in $M[G][K]$.\end{proof}

\begin{proof}[Proof of \autoref{contradiction-claim}] We will prove the requirements for $N:=j[N_0]$ one by one.

\emph{1, Cardinality and Containment of $\aleph_1$:} This is immediate from the second point of \autoref{lifting-claim} and the fact that $j(\aleph_1^{V[G]})=j(\aleph_1^V) = \aleph_1$.

\emph{2, Elementarity:} We use the Tarski-Vaught test in $M[G][K]$. Suppose we have $\bar{a} \in j[N_0]$ and $M[G][K] \models (\exists v \varphi(v,\bar{a}))^{H(j(\nu))}$. Let $\bar{b} \in N_0$ be such that $\bar{a} = j(\bar{b})$. So we are saying $M[G][K] \models  \exists \varphi(v,j(\bar{b}))^{H(j(\nu))}$, so by elementarity, we have that in $V[G]$, $H(\nu) \models \exists v \varphi(v,\bar{a})$. Let $c$ be a witness, i.e. $V[G] \models \varphi(c,\bar{a})^{H(\nu)}$ for $c \in N_0$. Therefore $j(H(\nu)) \models \varphi(j(c),j(\bar{b}))$ where $j(c) \in j[N_0]=N$.

\emph{3, Uniformity for $\aleph_{\omega+1}$:}  We want to show that in $M[G][K]$, $\cf(j[N_0] \cap j(\nu))=\aleph_1$. It is sufficient to observe that in $M[G][K]$, $\cf(j(\nu)) = \aleph_1$ since $\bL$ forces $\cf(\nu) \ge \aleph_1$ and this is preserved by the rest of the iteration.

\emph{4, Guessing:} Suppose for contradiction that in $M[G][K]$, there is an unbounded function $g:\omega_1 \to j(\nu)$ such that for all $i<\omega_1$, $g \rest i \in j[N_0]$. Then for all $i<\omega_1$, there is some $Y_i$ such that $j(Y_i) = g \rest i$. Let $f$ be the function $\bigcup_{i<\omega_1}Y_i$. We argue that $f$ is unbounded in $\nu$. If $\gamma<\nu$, then there is some $i<\omega_1$ such that $g(i) \in (j(\gamma),j(\nu))$. Then it follows by elementarity that $f(i) \in (\gamma,\nu)$. This therefore contradicts Condition \emph{(3)} from \autoref{lifting-claim}.\end{proof}

 Suppose $V[G] \models \textup{``}C \subseteq [H(\nu)]^{\aleph_1}\textup{''}$ is a club. Then given, lift $j:V[G] \to M[G][K]$ from \autoref{lifting-claim}. we observe that $j[C]$ is a directed subset of $j[N_0]$ because if $j(x),j(y) \in j[C]$, then we can find $z \in C$ such that $z \supseteq x \cup y$ and then $j(z) \supseteq j(x \cup y)$. Furthermore, for all $a \in j[N_0]$, there is some $x \in C$ such that $a \in j(x)$. Since $M[G][K] \models \textup{``}j(C)$ is a club in $j(H(\nu)^{V[G]})\textup{''}$, it follows that $j[N_0] \in j(C)$. The statement of \autoref{guessing-lemma} therefore follows by elementarity given the properties that we proved for $j[N_0]$.\end{proof}

\subsection{Failure of Weak Square} Now we prove:

\begin{lemma}\label{weak-square-fails} $V[G] \models \neg \square^*_{\aleph_\omega}$.\end{lemma}

We will use:

\begin{fact}[Fuchs-Rinot \cite{Fuchs-Rinot2018}]Suppose $\lambda$ is a singular cardinal such that $\mu^{\cf(\lambda)}<\lambda$ for all $\mu<\lambda$, and suppose $\square_\lambda^*$ holds. Then in a generic extension by $\textup{Add}(\lambda^+)$, there is a non-reflecting stationary subset of $\lambda^+ \cap \cof(\lambda)$.\end{fact}

Therefore it will be enough to show that if $\tilde{G}$ is $\Add(\aleph_{\omega+1})$-generic over $V[G]$, then $V[G][\tilde{G}] \models \textup{\textsf{Refl}}(\aleph_{\omega+1} \cap \cof(\omega))$, i.e$.$ that all stationary subsets of $\aleph_{\omega+1} \cap \cof(\omega)$ reflect.

We have an analogous pair of claims to deal with.

\begin{claim}\label{lifting-claim2} There is a forcing extension $V[G][\tilde{G}][K] \supset V[G][\tilde{G}]$ such that:

\begin{enumerate}

\item  $j:V \to M$ can be lifted to $j:V[G \ast \tilde{G}] \to M[j(G \ast H)]=M[G][\tilde{G}][K]$,

\item Stationary subsets of $\aleph_{\omega+1} \cap \cof(\omega)$ in $V[G][\tilde{G}]$ are stationary in $V[G][\tilde{G}][K]$,

\item $V[G][\tilde{G}][K] \models \textup{``}\cf(\nu) = \aleph_1\textup{''}$.
\end{enumerate}\end{claim}

\begin{claim}\label{contradiction-claim2} Suppose we have the lift $j:V[G \ast \tilde{G}] \to M[j(G * \tilde{G})]$ with the properties declared in \autoref{lifting-claim2}. Assume also that $|\nu|=\aleph_1$ in $V[G][\tilde{G}]$. Then $V[G][\tilde{G}] \models \textup{\textsf{Refl}}(\aleph_{\omega+1} \cap \cof(\omega))$.\end{claim}

Now we can prove the claims.

\begin{proof}[Proof of \autoref{lifting-claim2}] We choose an embedding $j$ with $j(\ell)(\kappa) = \dot{\S} \ast \dot{\G} \ast \dot{\Add}(\nu) \ast \dot{\bL}$. We will lift the embedding in a series of steps in which we extend the range of an embedding to the next generic. We will argue for point \emph{(2)} in the first step.

\emph{Lifting to domain $V[G_{\mathbb{I}}]$:} Because of Case (2) of the definition of $\mathbb{I}$, we have $V[G_\mathbb{I}] \models \WRP$. Therefore $\S \ast \dot{\G} \ast \dot{\Add}(\aleph_{\omega+1}) \ast \dot{\bL}$ is semiproper in $V[G_{\mathbb{I}}]$. By the chain condition of $\mathbb{I}$, we have $M[G_{\mathbb{I}}]^\nu \subset M[G_{\mathbb{I}}]$, so $\S \ast \dot{\G} \ast \dot{\Add}(\nu) \ast \dot{\bL}$ is also semiproper in $M[G_{\mathbb{I}}]$. Therefore, by elementarity, we have that
\[
j(\mathbb{I}) = \mathbb{I} \ast \dot{\S} \ast \dot{\G}  \ast \dot{\Add}(\aleph_{\omega+1}) \ast \dot{\mathbb{L}}  \ast \dot{\E}.
\]

Let $K_\mathbb{L}$ be $\mathbb{L}$-generic over $V[G][\tilde{G}]$ and let $K_{\E}$ be $\E$-generic over $V[G][K_{\mathbb{L}}]$. Then Silver's classical lifting argument (see \cite{Handbook-Cummings}) gives us an embedding $j:V[G_{\mathbb{I}}] \to M[G_{\mathbb{I}} \ast G_\S \ast G_\G \ast G_{\mathbb{L}} \ast \tilde{G} \ast K_{\bL} \ast K_{\E}] = M[j(G_{\mathbb{I}})]$ which is defined in $V[j(G_{\mathbb{I}})]$. (We will drop the dots to refer to the evaluated forcings in $V[j(G_\mathbb{I})]$.)

Now we argue for stationary preservation:  We know that $\mathbb{L}$ forcing preserves stationary subsets of $\aleph_{\omega+1} \cap \cof(\omega)$. We also know that that $\mathbb{L}$ collapses $\aleph_{\omega+1}^W$ to have cardinality $\aleph_1$ by \autoref{cummings-collapse}. It will be sufficient, therefore, to show that stationary subsets of $\omega_1$ are preserved in the remaining steps.

\emph{Lifting to domain $V[G_{\mathbb{I}}][G_{\S}]$}:  Analogous to \autoref{lifting-claim} in terms of both obtaining the lift and showing that the extension preserves stationary subsets of $\omega_1$.

\emph{Lifting to domain $V[G_{\mathbb{I}}][G_\S][G_{\G}]$}: Analogous to \autoref{lifting-claim}.

\emph{Lifting to domain $V[G_{\mathbb{I}}][G_\S][G_{\G}][\tilde{G}]$}: Use a master condition argument. As in the previous steps, we have $j"\tilde{G} \in M[j(G_{\mathbb{I}})]$. Therefore $q = \bigcup j"\tilde{G} \in M[j(G_{\mathbb{I}})]$, so we let $K_\textup{Add}$ be any $j(\Add(\nu))$-generic containing $q$. As in the previous two steps, $j(\Add(\nu))$ is countably closed and therefore preserves the relevant stationary sets.

Then we let $K = G_{\mathbb{L}} \ast G_{\E} \ast G_\S \ast G_\G$, completing the proof of the claim.\end{proof}

\begin{proof}[Proof of \autoref{contradiction-claim2}] (See \cite[Claim 7]{Cummings-Foreman-Magidor2003}.) Let $S \in V[G]$ be a stationary subset of $\aleph_{\omega+1} \cap \cof(\omega)$. Consider the lift $j:V[G][\tilde{G}] \to M[G][\tilde{G}][K]$ from \autoref{lifting-claim2}. Again, let $\rho = \sup j[\nu]$ where $\nu=\kappa^{+\omega+1}$. It is straightforward to show that $M[G][\tilde{G}][K] \models \textup{``}S \cap \rho$ is stationary $\textup{''}$, from which the claim follows by elementarity.\end{proof}

This finishes the lemma on the failure of $\square_{\aleph_\omega}^*$ and it completes the proof of \autoref{maintheorem}.

\subsection{Non-Tightness of the Models from the Main Theorem}

We include a last observation on the model from \autoref{maintheorem}. Recall \autoref{tightness-def}.

\begin{proposition} In $V[G]$, for all $\theta \ge \aleph_{\omega+1}$, there are stationarily many $N \prec H(\theta)$ of cardinality $\aleph_1$ that are not tight for $K = \{\aleph_n:2 \le n < \omega\}$.\end{proposition}

\begin{proof} Let $j[N_0]$ be as in \autoref{lifting-claim} and \autoref{contradiction-claim}. It is enough to show that $j[N_0]$ is not tight. Without loss of generality, we can assume that a fixed club $C \subseteq [H(\nu)^{V[G]}]^{\aleph_1}$ is chosen such that $\forall x \in C$, if $\cf(\sup(\nu \cap x)) = \omega_1$, then $\sup(\nu \cap x)$ is a good point of the scale $\vec{f}$ that was added by $\G$. This uses a predicate $[\alpha \mapsto f_\alpha]$ for $\vec{f}$ in the structure $H(\nu)^{V[G]}$ (since we cannot literally have $\vec{f} \in H(\nu)^{V[G]}$ because $|\vec{f}|=\nu$).

Recall again that by point (3) of definition \autoref{laver-namba-def-rep} and the preservation of $\aleph_1$, $\cf(j(\vec f))_\rho)(k) = \aleph_1$ for sufficiently large $k<\omega$.

First, we argue that $j(b_{\mathsf{prod}})$ is an exact upper bound of $j(\vec{f}) \rest \rho$ (using the same abuse of notation from earlier). This follows by the argument for \autoref{semiproper-bounding-claim}.

Since $j[N_0] \cap j(\nu) = \rho$, no element of $j[N_0]$ can dominate $j(b_\mathsf{prod})$: Suppose $g \in j[N_0] \cap j(\nu)$. Then there is $\alpha<\nu$, $g<^* f^*_{j(\alpha)} <^* j(b_\mathsf{prod})$.

We conclude by noting that $j(b_\mathsf{prod}) \in \prod_{n<\omega}(j(\aleph_n^{V[G]}) \cap j[N_0])$.\end{proof}

\section{A Variation of the Main Theorem}

\subsection{Another Failure of Square for the Earlier Models}

We will define a ``Miller version'' of our Namba forcing, which is used by Cummings, Foreman, and Magidor as well as Krueger.

\begin{definition}\label{miller-namba-def-rep} Fix a function $d: \omega \to \omega \setminus \{0,1\}$ as in \autoref{laver-namba-def-rep}.

The poset $\M$ will consist of trees $p$ such that the following hold:

\begin{enumerate}

\item $p$ is a tree consisting of finite sequences $t$.

\item For all $t \in p$ and $n \in \dom(t)$, $t(n) \in \aleph_{d(n)}$.

\item For all $t \in p$, there is some $t' \sqsupseteq t$ such that if $n = \dom(t')$, then $\{\eta:  t {}^\frown \eta \in p \}$ has cardinality $\aleph_{d(n)}$.

\end{enumerate}

The ordering on $\M$ is given by inclusion.\end{definition}
 
We will not use $\M$ in conjunction with good scales, so there is no need to employ stationary splitting.

We take an interest in a different type of square sequence:

\begin{definition} If $\lambda$ is a regular cardinal and $\kappa \le \lambda$, then $\square(\lambda,\kappa)$ holds if there is a sequence $\seq{\mathcal{C}_\alpha}{\alpha<\lambda}$ such that:
\begin{enumerate}
\item for all $C \in \mathcal{C}_\alpha$, $C$ is a club in $\alpha$,
\item for all $C \in \mathcal{C}_\alpha$ and $\beta \in \lim C$, $C \cap \beta \in \mathcal{C}_\beta$,
\item there is no club $D \subseteq \lambda$ such that for all $\alpha \in \lim D$, $D \cap \alpha \in \mathcal{C}_\alpha$.
\end{enumerate}\end{definition}

In this subsection, we will obtain:

\begin{theorem}\label{little-theorem} Assuming the consistency of a supercompact cardinal, there is a model in which the following hold:

\begin{enumerate}
\item $\aleph_\omega$ is a strong limit,
\item $\square_{\aleph_n}$ holds for all $n<\omega$,
\item $\square(\aleph_{\omega+1},\aleph_1)$ fails.
\end{enumerate}\end{theorem}

The point we are making is that we obtain a failure of a variant of the square principle that probably cannot be extracted from the failure of simultaneous reflection. Work of Hayut and Lambie-Hanson shows that (in particular) $\square(\aleph_{\omega+1},\aleph_1)$ is consistent with simultaneous reflection for $\aleph_0$-many subsets of $\aleph_{\omega+1}$ \cite[Theorem 4.11]{Hayut-LambieHanson2017}. Moreover, we make use of an approximation property that was not employed in the earlier papers. One could also make $\square_{\aleph_{\omega}}^*$ fail, as Krueger does \cite{Krueger2013}.

\subsubsection{More Approximation}

\begin{theorem}\label{long-approx-thm} Let $\dot{\U}$ be a $\M$-name for a countably closed forcing. Then if
\[
\Vdash_{\M \ast \dot{\U}} \textup{``}\dot{X}\textup{ is unbounded in }\aleph_{\omega+1}^V \textup{ and }\dot{X} \notin V\textup{''},
\]
then $\M \ast \dot{\U}$ forces that there is some $\delta<\aleph_{\omega+1}^V$ such that $\dot{X} \cap  \delta  \notin V$.\end{theorem}

\begin{proof} Let $\nu = \aleph_{\omega+1}^V$. Suppose for contradiction that $\dot{X}$ is a $\M$-name for a new cofinal subset of $\nu$, all of whose initial segements are in $V$.

Let $\varphi(\delta,q,\dot{d})$ denote the formula
\begin{align*}
\delta< \nu \wedge & (q,\dot{d}) \in \M \ast \dot{\U} \wedge \exists \seq{a_\alpha}{\alpha \in \osucc_q(\stem(q))} \su \\
&\forall \alpha \in \osucc_q(\stem(q)),(q \rest (\stem(q) {}^\frown \langle \alpha \rangle),\dot{d}) \Vdash ``\dot{X} \cap \delta = a_\alpha\textup{''} \wedge \\
& \forall \alpha, \beta \in \osucc_q(\stem(q)), \alpha \ne \beta \then a_\alpha \ne a_\beta.
\end{align*}

\begin{claim}\label{kurepa-claim} Let $\kappa$ be a cardinal and let $\lambda$ be a regular cardinal such that $\kappa^+ < \lambda$. Define
\[
T= \{t \in {}^{<\lambda}2: \exists q \le p , q \Vdash \textup{``}\dot{X} \cap \dom(t) = t \textup{''}\}.
\]
Then for all $p \in \M$, there is a club $C \subseteq \lambda$ such that the level $T_\delta$ has width $>\kappa$ for all $\delta \in C \cap \cof(\kappa^+)$.\end{claim}

\begin{proof} Let $T$ be the tree defined in the statement of the proposition. Suppose contrapositively that there is a stationary subset $S \subseteq \lambda \cap \cof(\kappa^+)$ such that for all $\delta \in S$, $|T_\delta| \le \kappa$. By the proof of a theorem of Kurepa \cite[Lemma 2.7]{Todorcevic-HOST}, it must be the case that for all $\delta \in S$ and $t \in T_\delta$, there is some $t' \sqsubseteq t$ such that $\dom(t') < \dom(t)$ and such that $t'$ has a unique descendant in $T_\delta$.

Now define a regressive function $h_0:S \to \lambda$ such that for all $\delta \in S$, there is some $t \in T_\delta$ and some $t' \sqsubseteq t$ with $\dom(t') = h_0(\delta)<\delta$. By Fodor's Lemma, there is some $S' \subseteq S$ and some $\gamma$ such that for all $\delta \in S'$, there is some $t \in T_\delta$ with a predecessor on level $\gamma$ with only $t$ as a successor in $T_\delta$. By considering limit points of $S$, we can assume that the predecessor is on a level $\gamma$ with $\gamma \in S$, and hence that there are at most $\kappa$-many to choose from. Hence we can apply Fodor again to find some $S'' \subseteq S'$ and some $\bar{t}$ such that for all $\delta \in S''$, there is some $t \in T_\delta$ such that $\bar{t} \sqsubseteq t$ and $t$ is the unique descendant of $\bar{t}$ in $T_\delta$.

Let $\bar{\delta}=\dom(\bar{t})$ and let $(q,\dot{d}) \le (p,\dot{c})$ decide $\bar{t} = \dot{X} \cap \bar{\delta}$. Since $\bar{t} \in T$, $(q,\dot{d}) \Vdash \textup{``} \bar{t}= \dot{X} \cap \bar{\delta}\textup{''}$. Let $t_\delta$ be the unique descendant of $\bar{t}$ in $T_\delta$. For all $\gamma<\lambda$, the fact that $\dot{X}$ is forced to be unbounded unbounded implies that $(q,\dot{d}) \Vdash \textup{``}\exists \delta \in (\gamma,\lambda),\dot{X} \cap (\gamma,\lambda) \ne \emptyset\textup{''}$. Therefore for all $\gamma<\lambda$, there is some $\delta \in (\gamma,\lambda)$ such that $\sup t_\delta \in (\gamma,\lambda)$ by uniqueness of $t_\delta$. Otherwise, $(q,\dot{d})$ would force that $\dot{X}$ is bounded in $\lambda$. Moreover, for these values $\gamma,\delta$ we have $(q,\dot{d}) \Vdash \textup{``}\dot{X} \cap \gamma = t_\delta \cap \gamma\textup{''}$. Hence $(q,\dot{d}) \Vdash \textup{``}\dot{X} = \bigcup_{\delta \in S''}t_\delta\textup{''}$. This contradicts the fact that $\dot{X}$ is forced to be new.\end{proof}

\begin{claim}\label{basic-claim-miller}  $\forall  \gamma < \nu,(p,\dot{c}) \in \M$, there is some $\delta \in (\gamma,\nu)$, and some $(q,\dot{d}) \le (p,\dot{c})$ with $\stem q = \stem p$ such that $\varphi(\delta,p,\dot{c})$ holds.\end{claim}

\begin{proof} Let $W = \osucc_p(\stem(p))$.

For each $\alpha \in W$, let
\[
T_\alpha = \{t \in {}^{<\lambda}2: \exists (q,\dot{d}) \le (p \rest (\stem(p) {}^\frown \langle \alpha \rangle),\dot{c}), (q,\dot{d}) \Vdash \textup{``}\dot{X} \cap \dom(t) = t\textup{''}\}.
\]
Let $C_\alpha$ be the set of $\delta$ such that $|T_\alpha| \ge \aleph_\omega$. Then $C_\alpha$ is a club by \autoref{kurepa-claim}. Let $C = \bigcap_{\alpha \in W}C_\alpha$. Fix some $\delta \in (C \cap \cof(|W|^+)) \setminus (\gamma+1)$.

By induction on $\alpha \in W$ we will define a sequence of conditions, $\seq{(q_\alpha,\dot{d}_\alpha)}{\alpha \in W}$ and the distinct sets $\seq{a_\alpha}{\alpha \in W}$.

Choose $(q_\alpha,\dot{d}_\alpha) \le (p \rest (\stem p {}^\frown \langle \alpha \rangle),\dot{c})$ forcing $\textup{``}\dot{X} \cap \delta = a_0\textup{''}$ for some $a_0$.

Now suppose that the members of our sequences have been defined for $\beta \in W \cap \alpha$. Let $B = \{a_\beta:\beta \in W \cap \alpha\}$, which is in particular of cardinality strictly less than $|W|$.

We do not have $ (p \rest t \concat \la \alpha \ra,\dot{c}) \Vdash \textup{``} \dot{X} \cap \delta \in B \textup{''}$, because this would contradict the fact that $\delta \in C_\alpha$. Therefore there is some $(q_\alpha,\dot{d}_\alpha) \le (p,\dot{c})$ and some $a_\alpha \notin B$ such that $(q_\alpha,\dot{d}_\alpha) \Vdash \textup{``}\dot{X} \cap \delta = a_\alpha\textup{''}$.

Then let $q = \bigcup_{\alpha \in W}q_\alpha$. Let $\dot{d}$ be the gluing of the $\dot{d}_\alpha$'s below $q_\alpha$.\end{proof}

Now we will build a condition $(q,\dot{d}) \in \M \ast \dot{\U}$ and an ordinal $\delta$ by a fusion process in such a way that any stronger condition deciding $\dot{X} \rest \delta$ will also code the generic sequence for $\M$.

Define a sequence $\seq{(p_n,\dot{c}_n)}{n<\omega}$ of conditions and a sequence $\seq{\gamma_n}{n<\omega}$ of ordinals in $\aleph_{\omega+1}$. 

The construction is defined as follows: Let $(p_0,\dot{c}_0)$ be obtained by applying \autoref{basic-claim-miller} to $(p,\dot{c})$.

In general that $(p_n,\dot{c}_n)$ and $\gamma_n$ have been defined. Consider all nodes $t$ of the $n\th$ splitting level of $p_n$. Then apply \autoref{basic-claim-miller} to $(p_n,\dot{c}_n)$ and $\gamma_n$ to obtain $(q_t,\dot{c}_t)$ and some $\gamma_t$. Let $p_{n+1} = \bigcup q_t$ and let $\dot{c}_{n+1}$ be the gluing of the $\dot{c}_t$'s. Let $\gamma_{n+1}$ be above the supremum of the $\gamma_t$'s.

At the end we let $q = \bigcap_{n<\omega}p_n$ be the fusion limit, we let $\dot{d}$ be the name forced to be a lower bound of the $\dot{c}_n$'s, and we let $\delta = \sup_{n<\omega}\gamma_n$.

Then $(q,\dot{d})$ forces that the generic sequence for $\M$ can be recovered from the forced value of $\dot{X} \rest \delta$ as follows: At the $0\th$ level, we choose a node based on the forced value for $\gamma_0$. Choose a $\sqsubseteq$-increasing sequences nodes inductively so that at step $n$, one is at the $n\th$ splitting level and checks the node corresponding to the forced value $\dot{X} \cap \gamma_n$.

Hence $(q,\dot{d}) \Vdash ``\dot{X} \rest \delta \notin V\textup{''}$ or else we obtain the contradiction from the previous paragraph. This contradicts the premise from the beginning of the proof that initial segments of $\dot{X}$ are in $V$.\end{proof}

\subsubsection{Relationship to Squares}

\begin{definition} Let $N$ be a set such that $\aleph_1 \subseteq N$ and assume that $\kappa$ is a regular cardinal. We say that $N$ is \emph{cofinally weakly $\kappa$-guessing} if for all unbounded $X \subseteq N \cap \kappa$ such that $X \cap \gamma \in N$ for $\gamma \in N \cap \kappa$, it follows that there is some $Y \in N$ such that $Y \cap (\sup(N \cap \kappa)) = X$.\end{definition}

\begin{proposition}\label{weak-guessing-square} Suppose that $\lambda$ is regular and $\theta > \lambda$. Suppose that there are stationarily-many $N \prec H(\theta)$ with $\cf(N \cap \lambda)=\aleph_1=|N|$ that cofinally weakly $\lambda$-guessing. Then $\square(\lambda,\omega_1)$ fails.\end{proposition}

\begin{proof}(See [Cox-Krueger, \cite{Cox-Krueger2018}].) Suppose for contradiction that we have $\vec{\mathcal{C}}=\seq{\mathcal{C}_\alpha}{\alpha<\lambda}$, a $\square(\lambda,\omega_1)$-sequence. Consider the structure $\mathcal{A} = (H(\theta),\in,<_\theta,\vec{\mathcal{C}})$. Let $N \prec \mathcal{A}$ be as in the hypothesis of the proposition. Let $\delta = \sup(N \cap \lambda)$. 

Take any $D \in \mathcal{C}_\delta$. By standard arguments, the fact that $\cf(N \cap \lambda)$ is uncountable implies that $N \cap \delta$ contains a club in $\delta$. (See e.g$.$ \cite[Lemma 4.3]{Cummings-Foreman-Magidor2004}.) Therefore $D \cap N \cap \delta$ is a club. Take some $\gamma \in D \cap N \cap \delta$. Then $D \cap \gamma \in \mathcal{C}_\gamma$, and since $\aleph_1 \subseteq N$, we have $\mathcal{C}_\gamma \subseteq N$, so therefore $D \cap \gamma \in N$. By unboundedness it follows that this will be true for any $\gamma \in N \cap \lambda$.

By the fact that $N$ is cofinally weakly $\lambda$-guessing, there is some $E \in N$ such that $E \cap \sup(N \cap \lambda) = D$. This implies that $N$ thinks that $E$ is a thread of the $\square(\lambda,\omega_1)$-sequence, which is a contradiction.\end{proof}

\subsubsection{Sketching the Rest of the Argument} The construction for \autoref{little-theorem} is very similar to that of \autoref{maintheorem} with four prominent differences: First, since we are not attempting to get all scales on $\aleph_\omega$ to be good, we will not use $\G(\prod_{2 \le n <\omega}\aleph_n)$, which makes the construction strictly easier. Second, we want to refer to models in $H(\theta)$ for $\theta \ge \aleph_{\omega+2}$ so that they may contain unbounded subsets of $\aleph_{\omega+1}$ as elements. Third, for similar reasons we are using a different sort of guessing---cofinal weak guessing for $\aleph_{\omega+1}$ rather than weak $(\omega_1,\aleph_{\omega+1})$-guessing. Fourth, we no longer need worry about the branches of the analog of $T_{\textup{IA}}$.

Let $g$ be $\S$-generic and let $k$ be $ \dot{\M}[g]$-generic over $V[g]$. Let $T_{\textup{IA}^+}$ be a the analog of $T_{\textup{IA}}$ for $H(\nu)^{V[g]} \cap ((H(\nu)^{V[g]})^{<\aleph_{\omega+1}})$ where the tree order is determined by end extension.

We start in a ground model $V$ in which $\kappa$ is a supercompact cardinal and fix a Laver supercompact guessing function $\ell:\kappa \to V_\kappa$.

We define a revised countable support iteration $\mathbb{I} = \seq{\mathbb{I}_\alpha,\dot{\mathbb{J}}_\alpha}{\alpha<\kappa}$ as follows:

\begin{enumerate}

\item Suppose $\alpha$ is inaccessible and that $\ell(\alpha)$ is an $\mathbb{I}_\alpha$-name for a poset of the form
\[
\dot{\S} \ast \dot{\M} \ast \dot{\B}(T_{\textup{IA}^+}) \ast \dot{\Col}(\aleph_1,\chi)
\]
where $T_{\textup{IA}^+}$ indicates the wide Aronszajn tree discussed above and $\chi$ is some regular cardinal. Then let $\dot{\mathbb{J}}_\alpha= \dot{\S} \ast \dot{\M} \ast \dot{\Col}(\aleph_1,\chi)$.

\item If $\alpha$ is inaccessible and $\ell(\alpha)$ is an $\mathbb{I}_\alpha$-name for $\Col(\aleph_1,\chi)$ for some $\chi<\kappa$, then let $\dot{\mathbb{J}}_\alpha$ be a name for $\Col(\aleph_1,\chi)$.

\item Otherwise $\dot{\mathbb{J}}_\alpha$ is a name for the trivial poset.

\end{enumerate}

As before we fix $\nu = \aleph_{\omega+1}^{V[G]}$.

The target model $V[G]$ will be a forcing extension by $\mathbb{I} \ast \dot{\S}$. To prove that $V[G] \models \neg \square(\aleph_{\omega+1},\aleph_1)$, we would argue for:

\begin{lemma}\label{guessing-lemma2} In $V[G]$, for all $\theta \ge \aleph_{\omega+1}$, there are stationarily-many $N \prec H(\theta)$ such that:
\begin{enumerate}
\item $\cf(N \cap \aleph_{\omega+1})=\aleph_1$,
\item $N$ is cofinally weakly $\aleph_{\omega+1}$-guessing.
\end{enumerate}\end{lemma}

This would be obtained from the following claims, where we note that the last point of each is changed for \autoref{little-theorem}.

\begin{claim}\label{2lifting-claim} There is a forcing extension $V[G][K] \supset V[G]$ such that:

\begin{enumerate}

\item  $j:V \to M$ can be lifted to $j:V[G] \to M[j(G)]=M[G][K]$,

\item $V[G][K] \models \textup{``}|\aleph_{\omega+1}^{V[G]}|=|\aleph_1^{V[G]}|=\aleph_1\textup{''}$,

\item In $M[G][K]$, if $f:\nu \to \nu$ is unbounded in $\aleph_{\omega+1}^{V[G]}$ and $f \rest \delta \in V[G]$ for all $\delta < \nu$, then $f \in V[G]$.

\end{enumerate}\end{claim}

\begin{claim}\label{2contradiction-claim} If $N_0 = H(\nu)^{V[G]}$, then in $M[G][K]$, the following hold:
\begin{enumerate}
\item $|j[N_0]| = \aleph_1 \subseteq j[N_0]$,
\item $j[N_0] \prec j(N_0)$,
\item $\cf(j[N_0] \cap j(\nu)) = \omega_1$,
\item $j[N_0]$ is cofinally weakly $\aleph_{\omega+1}$-guessing.
\end{enumerate}
\end{claim}

\subsection*{Comments on the Literature}

We would like to respectfully comment on some unresolved issues in the literature that are relevant to our work.

Part of our motivation came from the introduction of the paper of Cummings, Foreman, and Magidor in which they obtained the quasi-compactness of squares for $\aleph_\omega$. They had  set out to investigate, as they put it, ``the problem of the relationship between the sets of good and approachable points'' \cite[Page 2]{Cummings-Foreman-Magidor2004}. Their framework was the notion of \emph{canonical structure}, which refers to the infinitary objects whose definitions are essentially independent of any particular choices made in defining other objects \cite{Cummings-Foreman-Magidor2004,Cummings-Foreman-Magidor2006}. (In connection with approachability, see also \cite{Foreman-Magidor1997,Balogh-Davis-Just-Shelah-Szeptycki2000}.)

Some particulars of the model for \autoref{maintheorem} were fashioned after a statement of Cummings et al$.$ from their first canonical structures paper. The statement essentially went as follows \cite[Example 6.7]{Cummings-Foreman-Magidor2004}: Assume that $\aleph_\omega$ is a strong limit and that $2^{\aleph_\omega}=\aleph_{\omega+1}$. Assume also that $\vec{f}=\seq{f_\alpha}{\alpha<\aleph_{\omega+1}}$ is a continuous scale. Suppose that $S \in I[\aleph_{\omega+1} \cap \cof(\omega_1)]$. Then there is a club $C \subseteq \aleph_{\omega+1}$ such that if $N \prec H(\aleph_{\omega+2})$ has cardinality $\aleph_1$ and uniform cofinality $\omega_1$, $\sup(N \cap \aleph_{\omega+1}) = \gamma \in C\cap S$, and $\chi_N =^* f_{\sup(N \cap \aleph_{\omega+1})}$, then $N$ is internally approachable. The authors provided a sketch of the argument that did not appear require even internal unboundedness. It seemed that the conclusion was supposed to indicate something like sup-internal approachability.

However, Hannes Jakob pointed out to us that if there are stationarily many $N$ that are internally unbounded but not internally approachable (which is consistent \cite{Krueger2007}), then this stands as a counterexample to Example 6.7, which therefore cannot be literally true. To see this, observe that internally unbounded models of cardinality $\aleph_1$ are tight in $\prod_{n<\omega}\aleph_n$ and $\aleph_1$-uniform. It is then implied that the weak approachability ideal for $\aleph_{\omega+1}$ (see e.g$.$ \cite{Handbook-Eisworth}) is distinct from the approachability ideal, which is a contradiction if $\aleph_\omega$ is a strong limit.

Particularly in this context of these considerations, it is also natural to ask whether we can obtain the conclusion of \autoref{maintheorem} together with $\CH$. A reasonable approach would involve a version of the iteration presented above in \autoref{section-setup} that does not add reals to models of $\CH$. The problem of iterating Namba forcing without adding reals had been considered for a long time before being solved independently by Jensen, using the notion of subcomplete forcing \cite{Jensen2014}, and Shelah, using the notion of the $\mathbb{I}$-condition \cite[Chapter X]{PIF}. The trouble is in the difference between Laver-style and Miller-style Namba forcings, which are to some extent incompatible in iterations, as is demonstrated by a theorem of Magidor and Shelah \cite[Claim 4.2, Chapter XI]{PIF}. The $\mathbb{I}$-condition of Shelah applies to the Miller version, but our methods for \autoref{maintheorem} use the Laver version. Shelah did in fact announce results for the Laver version of the $\mathbb{I}$-condition (\cite[Remark XV.4.16A, Part 2]{PIF}, referencing Shelah $\#$ 311), but the paper never appeared.

\subsection*{Acknowledgments} We thank Hannes Jakob for providing the counterexample mentioned above and finding errors in the original versions of the paper, particularly in an erroneous argument for the claim that there is consistently a stationary set of structures that are sup-internally approachable but not internally unbounded. We thank James Cummings for some helpful correspondence, and we thank Grigor Sargsyan for clarifying the large cardinal hypotheses for failure of $\square_\kappa$ for limit $\kappa$.

\bibliographystyle{alpha}
\bibliography{bibliography}

\end{document}